\documentclass[12pt,reqno]{amsart}

\usepackage{hyperref}

\usepackage[all]{xy}

\xymatrixcolsep{2cm}

\numberwithin{equation}{section}

\allowdisplaybreaks[1]

\usepackage{etoolbox}

\makeatletter

\patchcmd{\thesubsection}{\arabic}{\Alph}{}{}

\patchcmd{\@seccntformat}{\@secnumfont}{%
\@secnumfont\expandafter\protect\csname format#1\endcsname}{}{}

\patchcmd{\@startsection}{\@afterindenttrue}{\@afterindentfalse}{}{}

\patchcmd{\subsection}{-.5em}{.3\linespacing}{}{}

\makeatother

\theoremstyle{plain}

\newtheorem{theorem}{Theorem}[section]
\newtheorem{question}[theorem]{Question}

\newtheorem{proposition}[theorem]{Proposition}

\newtheorem{corollary}[theorem]{Corollary}

\theoremstyle{remark}

\newtheorem{remark}[theorem]{Remark}

\newcommand{\Der}[3][]{\ensuremath{\mathrm{Der}_{#1} (#2, #3)}}

\newcommand{\Ker}[1]{\ensuremath{\mathrm{Ker} (#1)}}

\newcommand{\SKer}[1]{\ensuremath{\mathcal{K}er (#1)}}

\newcommand{\cat}[1]{\ensuremath{\mathcal{#1}}}

\newcommand{\at}[2][]{\ensuremath{\mathrm{at}_{#1} (#2)}}

\newcommand{\END}[2][]{\ensuremath{\mathcal{E}\mathit{nd}_{#1} (#2)}}

\newcommand{\id}[1]{\ensuremath{\mathbf{1}_{#1}}}

\renewcommand{\dim}[2][]{\ensuremath{\mathrm{dim}_{#1}(#2)}}

\newcommand{\N}{\ensuremath{\mathbb{N}}}

\newcommand{\R}{\ensuremath{\mathbb{R}}}

\newcommand{\C}{\ensuremath{\mathbb{C}}}

\newcommand{\ringed}[1]{\ensuremath{(#1,\,\struct{#1})}}

\newcommand{\struct}[1]{\ensuremath{\mathcal{O}_{#1}}}

\newcommand{\HOM}[3][]{%
\ensuremath{\mathcal{H}\mathit{om}_{#1}(#2,\, #3)}}

\newcommand{\DER}[3][]{%
\ensuremath{\mathcal{D}\mathit{er}_{#1}(#2,\, #3)}}

\newcommand{\DIFF}[4][]{\ensuremath{\mathcal{D}\mathit{iff}^{#1}_{#2}(#3,\,#4)}}

\newcommand{\Diff}[4][]{\ensuremath{\mathrm{Diff}^{#1}_{#2}(#3,\,#4)}}

\newcommand{\coh}[3]{\ensuremath{\mathrm{H}^{#1}(#2,\,#3)}}

\newcommand{\gcoh}[2]{\ensuremath{\mathrm{H}^{#1}(#2)}}

\newcommand{\drcoh}[3]{\ensuremath{\mathcal{H}}^{#1}_{#2}(#3)}

\newcommand{\hcoh}[3]{\ensuremath{\mathbb{H}^{#1}(#2,#3)}}

\begin{document}

\title[On the relative Connections]{On the relative connections}

\author[I. Biswas]{Indranil Biswas}

\address{School of Mathematics, Tata Institute of Fundamental Research,
Homi Bhabha Road, Mumbai 400005, India}

\email{indranil@math.tifr.res.in}

\author[A. Singh]{Anoop Singh}

\address{Harish-Chandra Research Institute, HBNI, Chhatnag Road, Jhusi,
Prayagraj 211019, India}

\email{anoopsingh@hri.res.in}

\subjclass[2010]{53C05, 14F10, 14H15}

\keywords{Relative connection, relative Atiyah exact sequence, K\"ahler manifold, Atiyah-Weil criterion.}

\begin{abstract}
We investigate relative connections on a sheaf of modules. A sufficient condition is given for the 
existence of a relative holomorphic connection on a holomorphic vector bundle over a complex analytic family. We show that the 
relative Chern classes of a holomorphic vector bundle admitting relative holomorphic connection vanish, if each of the fiber of the 
complex analytic family is compact and K\"ahler.
\end{abstract}

\maketitle

\section{Introduction}\label{Intro}

In \cite{A}, Atiyah introduced the notion of holomorphic connections in 
principal bundles over a complex manifold. A theorem due to Atiyah and Weil, \cite{A},
\cite{W} which is known 
as the \emph{Atiyah-Weil criterion}, says that a holomorphic vector bundle over a compact
Riemann surface admits a holomorphic connection if and only if the degree of each
indecomposable component of the holomorphic vector bundle is zero (see \cite{BR}
for an exposition of the Atiyah-Weil criterion); this criterion generalizes to holomorphic
principal bundles over a compact Riemann surface \cite{AB}.

We can ask following two basic questions:

\begin{question}\label{qn}
Let $\pi\,:\, X \,\longrightarrow\, S$ be a surjective holomorphic proper submersion
with connected fibers, and let $\varpi\,:\, E\,\longrightarrow\, X$ be a
holomorphic vector bundle.
\begin{enumerate}
\item \label{q:1} Suppose $E$ admits a relative holomorphic connection. Do the relative Chern classes vanish?

\item \label{q:2} Is there a good criterion for the existence of a relative holomorphic connection on 
$E$?
\end{enumerate} 
\end{question}

Here we shall answer the (first) Question \ref{qn}\eqref{q:1}, when each fiber of the complex
analytic family is compact and K\"ahler (See Theorem \ref{thm:3.18}). We are unable to
answer the (second) Question \ref{qn}\eqref{q:2}; instead we will answer the following question
(see Theorem \ref{thm:3l}):

\begin{question}
Let $\pi\,:\, X \,\longrightarrow\, S$ be a surjective holomorphic proper submersion such that 
$\pi^{-1}(t) \,=\, X_t$ are connected. Let 
$\varpi\,:\, E \,\longrightarrow\, X$ be a holomorphic vector bundle. Suppose that for every
$t \,\in\, S$, there is a holomorphic connection on holomorphic vector bundle
$\varpi|_{E_t}\,:\, E_t \,\longrightarrow\, X_t$. Does $E$ admit a relative holomorphic connection?
\end{question}

To answer the above questions, we shall develop theory of the relative connections on sheaf of 
modules. Let $(\pi,\,\pi^\sharp)\,:\,\ringed{X} \,\longrightarrow\, \ringed{S}$ be a morphism of ringed spaces and 
\cat{F} be an $\struct{X}$-module. In \cite{GD}, Grothendieck described the notion of $S$-derivation 
or relative derivation from $\struct{X}$ to $\cat{F}$. Also, in \cite{K1}, Koszul did ``Differential 
Calculus" in the frame work of commutative algebra, which can be reformulated in sheaf theoretic 
manner. We recall in Section \ref{sec:S-Conn}, from \cite{GD} and \cite{K1}, some preliminaries 
regarding relative derivations, relative connections, relative connections on associated 
\struct{X}-modules, covariant derivative and connection-curvature matrices in the setup of ringed 
spaces. In the third section, we define the relative differential operator between sheaves of 
\struct{X}-modules and symbol map of a first order relative differential operator. In Section 
\ref{sec:An-th}, we recall the notion of complex analytic family, which is done in \cite{KS1} in 
great detail. We establish the symbol exact sequence (see Proposition \ref{pro:4}).

\begin{proposition}[{Symbol exact sequence}]
Let $\pi \,:\, X \,\longrightarrow\, S$ be a holomorphic proper submersion of complex manifolds
with connected fibers, and
let $\cat{F}$, $\cat{G}$ be two locally free \struct{X}-modules of rank $r$ and $p$ respectively. Then the sequence
$$
0 \to \HOM[\struct{X}]{\cat{F}}{\cat{G}} \xrightarrow{\imath}\DIFF[1]{S}{\cat{F}}{\cat{G}} \xrightarrow{\sigma_1}\DER[S]{\struct{X}}{\HOM[\struct{X}]{\cat{F}}{\cat{G}}} \to 0
$$
is an exact sequence of $\struct{X}$-modules.
\end{proposition}

This (symbol) exact sequence is used in defining relative Atiyah algebra. The notion of relative 
holomorphic connections on an analytic coherent sheaf was introduced by Deligne \cite{D}, which 
coincides with the definition given here if it is formulated in the holomorphic setting. Given a 
relative holomorphic connection on a holomorphic vector bundle, there is an induced family of 
holomorphic connections; more precisely we prove the following (see Proposition \ref{pro:3J}):

\begin{proposition}
Let $\pi\,:\, X \,\longrightarrow\, S$ be a surjective holomorphic proper
submersion with connected fibers, and let $E \xrightarrow{\varpi} X$ be a 
holomorphic vector bundle. Suppose that we have a holomorphic $S$-connection 
$D$ on $E$. Then for every $t \in S$, we have a holomorphic connection $D_t$
on the holomorphic vector bundle $E_t \to X_t$. In other words, we have a
family $\{D_t\,\mid\, t \,\in\, S\}$ of holomorphic connections on the holomorphic 
family $\{\varpi\,:\, E_t \,\to\, X_t\,\mid\, t\,\in\, S\}$ of vector bundles.
\end{proposition} 

Next, we define the relative Chern classes of a complex vector bundle over
a complex analytic family, and prove the following (see Theorem \ref{thm:3.18}):

\begin{theorem}
Let $\pi\,:\, X \,\longrightarrow\, S$ be a surjective holomorphic proper
submersion, such that for each $t \in S$, $\pi^{-1}(t) = X_t$ is compact
connected K\"ahler manifold.
Let $E \,\xrightarrow{\varpi}\, X$ be a holomorphic vector 
bundle. Suppose that $E$ admits a holomorphic $S$-connection.
Then all the relative Chern classes $C_p^S(E) \,\in\, \drcoh{2p}{dR}{X/S}(S)$ 
of $E$ over $S$, are zero.
\end{theorem}

In last subsection \ref{suff}, we prove a sufficient condition for the existence of relative 
holomorphic connection. More precisely, we prove the following (see Theorem 
\ref{thm:3l}):

\begin{theorem}
Let $\pi\,:\, X \,\longrightarrow\, S$ be a surjective holomorphic proper submersion such that 
$\pi^{-1}(t) \,=\, X_t$ are connected. Let 
$\varpi\,:\, E\,\longrightarrow\, X$ be a holomorphic vector bundle. Suppose that for every
$t \in S$, there is a holomorphic connection on holomorphic vector bundle
$\varpi|_{E_t}: E_t \to X_t$, and
$$\coh{1}{S}{\pi_*(\Omega^1_{X/S}(\END[\struct{X}]{E}))} \,= \,0\, .$$
Then, $E$ admits a holomorphic $S$-connection.
\end{theorem}.

\section{S-Connections}\label{sec:S-Conn}

In this section, we shall introduce the notion of $S$-derivation following \cite{GD} and 
relative connection (or $S$-connection). Throughout this section we shall assume that 
\ringed{X}, \ringed{S} are two ringed spaces, while $(\pi,\,\pi^\sharp)\,:\,\ringed{X}
\,\longrightarrow\, \ringed{S}$ is a morphism between them.

\subsection{S-Linear morphism of sheaves}\label{sec:S-Lin-mor-sheaves:}

Let \cat{F}, \cat{G} be two $\struct{X}$-modules. A morphism
$\alpha : \cat{F} \to \cat{G}$ of sheaves of abelian groups is said to be
$S$-linear if for every open subset $V \subset S$, for every open subset
$U \,\subset\, \pi^{-1}(V)$, for every $t \,\in\, \cat{F}(U)$ and for every
$s \,\in \,\cat{O}_S(V)$, we have 
$$
\alpha(\rho_{U,\pi^{-1}(V)}(\pi^\sharp_V(s))t)\, =\, \rho_{U,\pi^{-1}(V)}
(\pi^\sharp_V(s)) \alpha(t)\, ,
$$
where $\rho_{U,\pi^{-1}(V)}\,:\, \mathcal{O}_X(f^{-1}(V)) \,\longrightarrow\, \mathcal{O}_X(U)$
is the restriction map.
We denote by $\HOM[S]{\cat{F}}{\cat{G}}$ the sheaf of $S$-linear morphism from $\cat{F}$ to $\cat{G}$. 
We denote $\rho_{U,\pi^{-1}(V)}(\pi^\sharp_V(s))$ by $s\vert _U$.

\subsection{S-derivation}\label{sec:S-der}

For the following definition see \cite{GD} (Chapitre IV, \textbf{16.5}).

Let $\cat{F}$ be an $\cat{O}_X$-module. An $S$-derivation from
$\cat{O}_X$ to \cat{F} is a morphism
$$
\delta: \cat{O}_X \to \cat{F}
$$
of sheaves of abelian groups which satisfies the following conditions:

\begin{enumerate}
 \item $\delta$ is an $S$-linear morphism.
 \item (Leibniz rule.) For every open subset $U \subset X$, and for
 every $a,b \in \cat{O}_X(U)$, we have 
$$
\delta_U(ab) \,=\, a \delta_U(b)+ \delta_U(a)b\, .
$$
\end{enumerate}
The set of all $S$-derivation from $ \cat{O}_X$ to $ \cat{F} $ form a left
$ \cat{O}_X(X)$-module denoted by $$\Der[S]{ \cat{O}_X}{\cat{F}}\, .$$ For every
open
subset $U \,\subset\, X$, we note that $\Der[S]{ \cat{O}_X \vert _U}{\cat{F} \vert _U}$ is a
left
$\cat{O}_X(U)$-module. For every open subset $U \,\subset\, X$, the assignment 
$U\,\longmapsto\, \Der[S]{ \cat{O}_X \vert _U}{\cat{F} \vert _U}$ is a sheaf of
$\struct{X}$-module
and it is denoted by $\DER[S]{\cat{O}_X}{\cat{F}}$.

Let $\END[S]{\cat{F}}$ denote the sheaf of $S$-linear endomorphism on
$\cat{F}$.
Then $\END[S]{\cat{F}}$ is an $\struct{X}$-module. In particular, if we take
$\cat{F} \,= \,\struct{X}$, then $\END[S]{\cat{O}_X}$ is a sheaf of Lie algebras
with respect to the bracket operation defined as follows:
$$[\xi,\eta] \,=\, \xi \circ \eta - \eta \circ \xi$$ for every open subset
$U\, \subset \, X$ and for all $\xi,\,\eta \,\in \,\END[S]{\cat{O}_X}(U)$.
We note that $\DER[S]{\struct{X}}{\struct{X}}$ is a Lie subalgebra of
$\END[S]{\struct{X}}$.

\begin{theorem}\label{thm:1}
Let $(\pi,\,\pi^\sharp)\,:\,\ringed{X} \,\longrightarrow\,\ringed{S} $ be a morphism of ringed
spaces. Then there exists a unique \struct{X}--module, denoted by
$\Omega^1_{X/S}$, and endowed with an (universal) $S$-derivation 
$$
d_{X/S}\,:\, \struct{X} \,\longrightarrow\, \Omega^1_{X/S} 
$$ 
satisfying following universal property:
For any $\struct{X}$-module $\cat{F}$, and for any $S$-derivation
$$
\delta\,:\, \struct{X} \,\longrightarrow\, \cat{F}\, ,
$$
there exists a unique $\struct{X}$-module homomorphism 
$$
\alpha\,:\, \Omega^1_{X/S} \,\longrightarrow\, \cat{F}
$$
such that 
$$
\delta \,=\, \alpha \circ d_{X/S}\, . 
$$
\end{theorem}

\begin{proof}
Cf. \cite{GD} (Proposition:16.5.3 on p.~28).
\end{proof}

\subsection{S-Connections}\label{Sec:S-con}

Let $ \mathcal{F}$ be an \struct{X}-module. An $S$-connection or \textbf{relative connection} on
$\mathcal{F}$ is an $ \struct{X}$--module homomorphism 
$$
 D \,: \,\DER[S]{\struct{X}}{\struct{X}} \,\longrightarrow\, \END[S]{\cat{F}}
$$
such that for every open subset $U$ of $X$ and for every
$\xi \,\in\, \DER[S]{\struct{X}}{\struct{X}}(U)$, the $\struct{X}(U)$--module
homomorphism 
$$D_U\, :\, \DER[S]{\struct{X}}{\struct{X}}(U) \,\longrightarrow\, \END[S]{\cat{F}}(U)\, ,
\ \ \xi \, \longmapsto \, (D_U)_\xi$$ satisfies the Leibniz rule which says that
$$
((D_U)_\xi)_V(ag) \,=\, \xi|_V(a)g + a ( (D_U)_\xi)_V(g)
$$
for every open subset $V$ of $U$, for all $a \,\in \,\struct{X}(V)$ and
$g \,\in \,\cat{F}(V)$.

If $\pi\,:\, X \,\longrightarrow\, S$ is a holomorphic map of complex 
manifolds with connected fibers, and $\cat{F}$ is a holomorphic vector bundle over $X$, we call
$D$ a \emph{holomorphic $S$-connection}. Similarly, If 
$\pi\,:\, X \,\longrightarrow\, S$ is a smooth map of smooth manifolds and $\cat{F}$ is a 
smooth vector bundle, we call $D$ a \emph{smooth $S$-connection}.

\begin{remark}\label{Rem:2}\mbox{}
\begin{enumerate}
\item
We note that $(D_U)_\xi\,:\, \cat{F}|_U \,\longrightarrow\, \cat{F}|_U$ is an $S$-linear endomorphism, where
$S$-linearity is with respect to $\pi|_U\,:\, U \,\longrightarrow\, S$.

\item The inclusion map
$$
\epsilon \,:\, \DER[S]{\struct{X}}{\struct{X}} \,\hookrightarrow\,
\END[S]{\struct{X}} 
$$
is an $S$-connection on the \struct{X}-module \struct{X}, and it is called
the \emph{canonical $S$-connection} on \struct{X}. To avoid the cumbersome
notation $(D_U)_\xi$, we shall simply denote it by $D_\xi$.
\end{enumerate} 
\end{remark}

\subsection{S-Connections on the associated \struct{X}-modules}\label{Sec:1.d}

Let $(\cat{F}_i)_{i \in I}$ be a family of $ \struct{X}$-modules, and for
each $i \in I$, let $D^i$ be an $S$-connection on $\cat{F}_i$. Then the various
\struct{X}-modules obtained from $(\cat{F}_i)_{i \in I}$ by functorial
construction has natural $S$-connections.

\subsubsection{Direct Sum}\label{Sec:D-Sum}

If
$$
\cat{F}\,=\, \bigoplus_{i\in I} \cat{F}_i\, ,
$$
and if we define 
$$
D_\xi(u) \,=\, (D_\xi^i(u_i))_{i\in I}
$$
for all sections $\xi$ of $\DER[S]{\struct{X}}{\struct{X}}$ and all
$u \,= \, (u_i)_{i \in I}$ of $ \cat{F}$, then we get an $S$-connection $D$ on
$\cat{F}$. In particular, if we take each $\cat{F}_i $ to be \struct{X} and
each $D^i$ to be $\epsilon$, the canonical $S$-connection on \struct{X},
then every free \struct{X}-module has a canonical $S$-connection.

\subsubsection{Tensor products}\label{Sec:T-prod}

Suppose $I \,=\, \{1,2,\cdots, p\}$, where $p$ is an integer $\geq 1$. Then for every open subset
$U$ of $X$, and for each $\xi \,\in\, \DER[S]{\struct{X}}{\struct{X}}(U)$, there
exists a unique $S$-linear endomorphism $D_\xi$ of
$\cat{F}_1\bigotimes_{\struct{X}} \cdots \bigotimes_{\struct{X}} \cat{F}_p$
such that on the presheaf level it is given by the formula
$$
D_\xi(s_1\otimes_{\struct{X}} \cdots \otimes_{\struct{X}}s_p) \,=\,
\sum^p_{i=1}s_1\otimes_{\struct{X}} \cdots \otimes_{\struct{X}}s_{i-1}
\otimes_{\struct{X}}D^i_\xi(s_i)\otimes_{\struct{X}}s_{i+1} 
\otimes_{\struct{X}} \cdots \otimes_{\struct{X}}s_p
$$
for every $s_1\otimes_{\struct{X}} \cdots \otimes_{\struct{X}}s_p \,\in\,
\cat{F}_1(U) \bigotimes_{\struct{X}(U)} \cdots \bigotimes_{\struct{X}(U)}
\cat{F}_p(U)$. This gives an $S$-connection on
$\cat{F}_1\bigotimes_{\struct{X}} \cdots \bigotimes_{\struct{X}}\cat{F}_p$.

Suppose that $\cat{F}_1 \,=\, \cat{F}_2 \,=\, \cdots \,= \,\cat{F}_p \,= \,\cat{F}$, and
denote
$\cat{F}_1\bigotimes_{\struct{X}} \cdots \bigotimes_{\struct{X}}\cat{F}_p$
by $T^p_{\struct{X}}(\cat{F})$. Equip $T^0_{\struct{X}}(\cat{F}) \,=\, \struct{X}$ with
the canonical $S$-connection $\epsilon$ on \struct{X}, and for each 
$p \geq 1$, equip $T^p_{\struct{X}}(\cat{F})$ with the $S$-connection
induced by the $S$-connection $D$ on $\cat{F}$; this
$S$-connection on $T^p_{\struct{X}}(\cat{F})$ will be denoted by $D^p$. Recall that the
tensor algebra of
$\struct{X}$-module \cat{F} is a graded \struct{X}-algebra
$$
T_{\struct{X}}(\cat{F}) \,=\, \bigoplus_{p \in \N}T^p_{\struct{X}}(\cat{F})\, .
$$
Let $D$ be the connection on the \struct{X}-module 
$T_{\struct{X}}(\cat{F})$,
which is the direct sum (see Section \ref{Sec:D-Sum}) of the connections. It is called the induced
connection on $T_{\struct{X}}(\cat{F})$.

\begin{remark}
\label{Rem:3}
On tensor algebra $T_{\struct{X}}(\cat{F})$, we have
$$
D_\xi(s\otimes t) = D_\xi(s)\otimes t + s\otimes D_\xi(t),
$$
for all local sections $\xi$ of \DER[S]{\struct{X}}{\struct{X}} and local
sections $s,t$ of $T_{\struct{X}}(\cat{F})$.
\end{remark}

\subsubsection{Submodule and quotient module}\label{Sec:S-mod-Q-mod}

Let \cat{F} be an \struct{X}-module with an $S$-connection $D$, and let $\cat{G}$
be
an \struct{X}-submodule of \cat{F}. Let \cat{H} denote the quotient
\struct{X}-module \cat{F}/\cat{G}. Suppose that for every section $\xi$ of
\DER[S]{\struct{X}}{\struct{X}}, we have 
$D_\xi(\cat{G}) \subset \cat{G}$. Then
$D$ will induce an $S$-connection on $\cat{G}$ and on $\cat{H}$.

\subsubsection{Symmetric algebra and exterior Algebra}\label{Sym-Ext-Alg}

Let $D$ be an $S$-connection on \cat{F}. The
$S$-connection on the tensor algebra $T_{\struct{X}}(\cat{F})$ induced by
$D$ will also be denoted by $D$. Let $\cat{I}$
denote the two sided ideal sheaf of $T_{\struct{X}}(\cat{F})$ described as
follows:

for every open subset $U$ of $X$, let $\cat{I}(U)$ be the two sided ideal in
$T_{\struct{X}}(\cat{F})(U)$ generated by elements of the form 
$s \otimes t - t \otimes s$, where $s,\,t \,\in\, \cat{F}(U)$. 
Then $D_\xi(\cat{I}) \,\subset\, \cat{I}$ for all sections $\xi$ of
$\DER[S]{\struct{X}}{\struct{X}}$. Thus, by above Section 
\ref{Sec:S-mod-Q-mod}, we get an $S$-connection $D$ on the symmetric algebra
$$
Sym_{\struct{X}}(\cat{F})\, =\, T_{\struct{X}}(\cat{F})/\cat{I}
$$
of $\cat{F}$.

Similarly, let \cat{J} denote the two sided ideal sheaf of 
$T_{\struct{X}}(\cat{F})$ generated by the local sections $s\otimes s$ of
$T^2_{\struct{X}}(\cat{F})$, where $s$ is a local section of \cat{F}. Then 
we have $D_\xi(\cat{J}) \,\subset \,\cat{J}$ for all local sections $\xi$ of 
\DER[S]{\struct{X}}{\struct{X}}, and hence a connection on the exterior
algebra 
$$
\Lambda_{\struct{X}}(\cat{F}) \,=\, T_{\struct{X}}(\cat{F})/\cat{J}
$$
of \cat{F} is obtained.

\begin{remark}\label{Rem:4}\mbox{}
\begin{enumerate}
\item For all $p\,\in\, \N$, we have 
$D_\xi(Sym^p_{\struct{X}}(\cat{F})) \,\subset\, Sym^p_{\struct{X}}(\cat{F})$,
where $Sym^p_{\struct{X}}(\cat{F})$ is the $p$-th graded component of the symmetric
algebra $Sym_{\struct{X}}(\cat{F})$. Consequently, we get an $S$-connection on 
$Sym^p_{\struct{X}}(\cat{F})$. Similarly, we get an $S$-connection on 
$\Lambda^p_{\struct{X}}(\cat{F})$.
 
\item We have
$$
D_\xi(ss') \,=\, D_\xi(s)s' + sD_\xi(s')
$$
for all local sections $\xi$ of $\DER[S]{\struct{X}}{\struct{X}}$ and 
$s,s'$ of $Sym_{\struct{X}}(\cat{F})$, and
$$
D_\xi(t \wedge t') \,=\, D_\xi(t) \wedge t' + t \wedge D_\xi(t')
$$
for all local sections $\xi$ of $\DER[S]{\struct{X}}{\struct{X}}$ and 
$t,t'$ of $\Lambda_{\struct{X}}(\cat{F})$.
\end{enumerate}
\end{remark}

\subsubsection{S-Connection on \HOM[\struct{X}]{\cat{F}}{\cat{G}}}\label{Hom(F,G)}

Let \cat{F}, \cat{G} be \struct{X}-modules with $S$-connections 
$D^{\cat{F}}$
and $D^{\cat{G}}$ respectively. For every local section $\xi$ of
$\DER[S]{\struct{X}}{\struct{X}}$, let $D_\xi$ be the $S$-linear endomorphism
of
the \struct{X}-module \HOM[\struct{X}]{\cat{F}}{\cat{G}},which is defined by
$$
D_\xi(h) = D^{\cat{G}}_\xi \circ h - h \circ D^{\cat{F}}_\xi,
$$
for all local sections $h$ of \HOM[\struct{X}]{\cat{F}}{\cat{G}}. Then the
morphism
$$D = \xi \mapsto D_\xi : \DER[S]{\struct{X}}{\struct{X}} \to
\END[S]{\HOM[\struct{X}]{\cat{F}}{\cat{G}}}$$
is an $S$-connection on $\HOM[\struct{X}]{\cat{F}}{\cat{G}}$.

\begin{remark}
\label{Rem:5}\mbox{}
\begin{enumerate}
\item If $\cat{F}\,=\, \cat{G}$, and $ D^\cat{F}\,=\, D^\cat{G}$, then the above $S$
-connection $D$ on \END[S]{\cat{F}} is given by 
$$
D_\xi(h) \,=\, [D^{\cat{F}}_\xi,\, h] \,=\, D^{\cat{G}}_\xi \circ h - h \circ
D^{\cat{G}}_\xi
$$
for all local sections $h$ of $\END[S]{\cat{F}}$.

\item If $\cat{G} \,=\, \struct{X}$, and if $D^{\cat{G}}$ is the
canonical connection on $\struct{X}$, then the above connection on 
$\cat{F}^* \,=\, \HOM[\struct{X}]{\cat{F}}{\struct{X}}$ is given by
$$
D_\xi(f) \,=\, \xi \circ f - f \circ D^{\cat{F}}_\xi
$$
for all local sections $f$ of $\cat{F}^*$.
\end{enumerate}
\end{remark}

\subsubsection{S-Connection on \struct{X}-module of \struct{X}-multilinear maps}
\label{S-Con-Mult}

Let $p \,\geq \,1$ be an integer, and let 
$\cat{F}_1,\,\cat{F}_2,\, \cdots,\, \cat{F}_p,\, \cat{G}$ be \struct{X}-modules with
$S$-connections $D^1,\,D^2,\,\cdots,\,D^p,\,D^{\cat{G}}$ respectively. For every
open subset $U$ of $X$, define
$$
\cat{L}_{\struct{X}}(\cat{F}_1,\cdots,\cat{F}_p;\cat{G})(U) \,:=\,
L_{\struct{X}|_U}(\cat{F}_1|_U,\cdots,\cat{F}_p|_U;\cat{G}|_U)\, ,
$$
where $ L_{\struct{X}|_U}(\cat{F}_1|_U,\cdots,\cat{F}_p|_U;\cat{G}|_U)$ is
the $\struct{X}(U)$-module of $\struct{X}|_U$-multilinear maps from 
$\cat{F}_1|_U \times \cat{F}_2|_U \times \cdots\times \cat{F}_p|_U$ to $\cat{G}|_U$. 
The sheaf of \struct{X}-multilinear maps is denoted by $\cat{L}_{\struct{X}}(\cat{F}_1,\cdots,\cat{F}_p;\cat{G}$).
For every local
section $\xi$ of \DER[S]{\struct{X}}{\struct{X}}, let $D_\xi$ be the $S$-linear
endomorphism of the \struct{X}-module 
$\cat{L}_{\struct{X}}(\cat{F}_1,\cdots ,\cat{F}_p;\cat{G})$ defined by
$$
D_\xi(\omega)(u_1,u_2,\ldots,u_p) \,=\, 
D^\cat{G}_\xi(\omega(u_1,u_2,\ldots,u_p)) 
- \sum^p_{i = 1} \omega(u_1,\ldots,u_{i-1},D^i_\xi(u_i),u_{i+1},
\ldots, u_p)\, ,
$$
for all local sections $\omega $ of
$\cat{L}_{\struct{X}}(\cat{F}_1,\cdots,\cat{F}_p;\cat{G})$ and local
sections $(u_1,\,u_2,\,\cdots,\,u_p)$ of 
$\cat{F}_1 \times \cat{F}_2 \times \cdots\times \cat{F}_p$. Then the
morphism
$$D \,= \,\xi \,\longmapsto\, D_\xi \,:\, \DER[S]{\struct{X}}{\struct{X}}
\,\longrightarrow\, \END[S]{\cat{L}
_{\struct{X}}(\cat{F}_1,\cdots,\cat{F}_p;\cat{G}})$$ is an $S$-connection on 
$\cat{L}_{\struct{X}}(\cat{F}_1,\cdots,\cat{F}_p;\cat{G})$.
 
\begin{remark}\label{Rem:6} 
Let $\cat{L}^p_{\struct{X}}(\cat{F},\cat{G})$ denote the
$\struct{X}$-module 
$\cat{L}_{\struct{X}}(\cat{F}_1,\cdots,\cat{F}_p;\cat{G})$,
 where $\cat{F}_1 \,=\, \cdots\,=\, \cat{F}_p \,=\, \cat{F}$. Let $D$ be the
$S$-connection on $\cat{L}^p_{\struct{X}}(\cat{F},\cat{G})$ induced by
$D^{\cat{F}}$ and $D^{\cat{G}}$.
Let $\cat{S}ym^p_{\struct{X}}(\cat{F},\cat{G})$ (respectively,
$\cat{A}lt^p_{\struct{X}}(\cat{F},\cat{G})$) denote the $\struct{X}$-submodule
of $\cat{L}^p_{\struct{X}}(\cat{F},\,\cat{G})$ consisting of symmetric
(respectively, alternating) \struct{X}-multilinear maps from $\cat{F}^p$ to 
 $\cat{G}$. 
 Then $$D_\xi(\cat{S}ym^p_{\struct{X}}(\cat{F},\cat{G})) \,\subset\, \cat{S}
 ym^p_{\struct{X}}(\cat{F},\cat{G})$$
 (respectively, $D_\xi(\cat{A}lt^p_{\struct{X}}(\cat{F},\cat{G})) \,\subset\,
 \cat{A}lt^p_{\struct{X}}(\cat{F},\cat{G})$). Therefore, 
 $D$ induces an $S$-connection on the \struct{X}-submodules 
 $$\cat{S}ym^p_{\struct{X}}(\cat{F},\cat{G})\ \ \text{ and }\ \
\cat{A}lt^p_{\struct{X}}(\cat{F},\cat{G})$$ of 
 $\cat{L}^p_{\struct{X}}(\cat{F},\cat{G})$.
 \end{remark}

\subsubsection{Compatibility of multilinear maps and $S$-connections}\label{comp-mult-conn}

Let $p \,\geq\, 1$, and let $$\cat{F}_1,\,\cdots,\,\cat{F}_p,\,\cat{G}$$ be
$\struct{X}$-modules with $S$-connections $ D^1,\,\cdots,\,D^p,\,D^{\cat{G}}$ respectively. 
Let
$$
 \mu \,:\, \cat{F}_1 \times \cat{F}_2 \times \cdots\times \cat{F}_p \,\longrightarrow\, \cat{G}
$$
be an \struct{X}-multilinear map. We say that 
$D^1,\,D^2,\,\cdots,\,D^p,\,D^{\cat{G}},\, \mu$ are compatible if for every local
section $\xi$ of $\DER[S]{\struct{X}}{\struct{X}}$, and for all local sections
$(u_1,\,u_2,\,\cdots,\,u_p)$ of $\cat{F}_1 \times \cat{F}_2 \times \cdots\times
\cat{F}_p$, we have
$$
 D^{\cat{G}}_\xi(\mu(u_1,\cdots,u_p)) = \sum^p_{i = 1} \mu(u_1,
 \cdots,u_{i-1},D^i_\xi(u_i),u_{i+1},\ldots,u_p)\, .
$$

The following proposition is straight-forward to prove.

\begin{proposition}\label{pr1}
Let $\cat{F}$, $\cat{G}$ and $\cat{H}$ be \struct{X}-modules, and let
$$ \mu\,:\, \cat{F} \times
 \cat{G} \,\longrightarrow\, \cat{H}$$ be a $\struct{X}$-bilinear map. Let $\cat{K}$ be any
$\struct{X}$-module and $p \,\geq\, 1$, $q \,\geq\, 1$. Then, we have a $\struct{X}$-bilinear map
$$
\wedge\,:\,\cat{A}lt^p_{\struct{X}}(\cat{K},\cat{F}) \times 
\cat{A}lt^q_{\struct{X}}(\cat{K},\cat{G})\,\longrightarrow\,
\cat{A}lt^{p+q}_{\struct{X}}(\cat{K},\cat{H})
$$
defined by
$$
\alpha \wedge \beta(u_1,\ldots,u_{p+q}) \,=\, 
\sum_{\sigma \in S(p,q)}sgn(\sigma)
\mu( \alpha(u_{\sigma(1)},\ldots,u_{\sigma(p)}), \beta(u_{\sigma(p+1)}
\ldots,u_{\sigma(p+q)}))
$$
for all local sections $\alpha$ of 
$\cat{A}lt^p_{\struct{X}}(\cat{K},\cat{F})$, and $\beta$ of 
$\cat{A}lt^q_{\struct{X}}(\cat{K},\cat{G})$, where $S(p,q)$ is the set of all 
$(p,q)$-shuffles, that is, the set of all permutation $\sigma \in S_{p+q}$ 
such that $\sigma(1)\,<\, \cdots\,<\, \sigma(p)$ and 
$\sigma(p+1) \,<\, \cdots\,<\, \sigma(p+q)$. 
\end{proposition}

The construction in Proposition \ref{pr1} produces the following corollary.

\begin{corollary}
Let $D^\cat{F}, D^\cat{G}$ and $D^\cat{H}$ be $S$-connections on 
$\cat{F},\cat{G}$ and \cat{H} respectively, which are compatible with 
$\mu$. Let $ D^\cat{K}$ be an $S$-connection on $\cat{K}$. Denote the induced
connections on 
$\cat{A}lt^p_{\struct{X}}(\cat{K},\cat{F})$, 
$\cat{A}lt^p_{\struct{X}}(\cat{K},\cat{G})$ and
$\cat{A}lt^{p+1}_{\struct{X}}(\cat{K},\cat{H})$ by $D^\cat{F}$, $D^\cat{G}$ 
and $D^\cat{H}$ respectively. Then $D^\cat{F}$, $D^\cat{G}$, $D^\cat{H}$
and $\wedge$ are compatible, that is,
$$
D^{\cat{H}}_\xi(\alpha \wedge \beta) \,=\, 
D^{\cat{F}}_\xi(\alpha) \wedge \beta + \alpha \wedge D^{\cat{G}}_\xi(\beta)
$$
for all local sections $\xi$ of $\DER[S]{\struct{X}}{\struct{X}}$, $\alpha$ of
$\cat{A}lt^p_{\struct{X}}(\cat{K},\cat{F})$ and $\beta$ of 
$\cat{A}lt^q_{\struct{X}}(\cat{K},\cat{G})$.
\end{corollary}

\subsection{The relative Lie derivative}\label{Sec:Lie-Der}

Let \cat{F} be an $\cat{O}_X$-module and $D$ an $S$-connection on
\cat{F}. Let $p \geq 1$ be an integer, $U$ an open subset of $X$, 
$\xi \in \DER[S]{\struct{X}}{\struct{X}}(U)$, and 
$\alpha \in \cat{L}^p_{\struct{X}}(\DER[S]{\struct{X}}{\struct{X}},\cat{F}
)(U)$. Then the map 
$$
\theta_\xi(\alpha)\,:\,(\DER[S]{\struct{X}}{\struct{X}}(U))^p \,\longrightarrow\, \cat{F}(U)\, ,
$$
defined by
$$
\theta_\xi(\alpha)(\eta_1,\cdots,\eta_p) \,=\, 
D_\xi(\alpha(\eta_1,\cdots,\eta_p))
- \sum^p_{i=1} \alpha(\eta_1,\cdots,\eta_{i-1},[\xi,\eta_i],\eta_{i+1}
,\cdots,\eta_p)
$$
for all $\eta_1,\,\cdots,\,\eta_p \,\in \,\DER[S]{\struct{X}}{\struct{X}}(U)$, is an 
$\struct{X}(U)$-multilinear map. Moreover, the map 
$$
\theta_\xi\,:\, 
\cat{F}^p_{\struct{X}}(\DER[S]{\struct{X}}{\struct{X}},\cat{F})(U)
\,\longrightarrow\, \cat{F}^p_{\struct{X}}(\DER[S]{\struct{X}}{\struct{X}},\cat{F})(U) 
$$
is $S$-linear, because $D_\xi$ is $S$-linear.

The $S$-linear morphism 
$$
\theta \,:\, \DER[S]{\struct{X}}{\struct{X}}\,\longrightarrow\, 
\END[S]{\cat{L}^p_{\struct{X}}(\DER[S]{\struct{X}}{\struct{X}},\cat{F})}
$$
is called the \textbf{relative Lie derivation} in degree $p$ associated with $D$.

\begin{remark}
\label{Rem:7}\mbox{}
\begin{enumerate}

\item The relative Lie derivation satisfies the followings:
\begin{itemize}
\item $\theta_\xi(\alpha + \beta) \,=\, 
\theta_\xi(\alpha) + \theta_\xi(\beta)$,

\item $\theta_\xi(a \alpha) \,=\, 
\xi(a) \alpha + a \theta_\xi(\alpha)$,

\item $\theta_{\xi + \zeta} (\alpha)\,=\, \theta_\xi(\alpha) +
\theta_\zeta(\alpha)$, and

\item $\theta_{s\xi}(\alpha) \,=\, s \theta_\xi(\alpha)$
\end{itemize}
for all local sections $\alpha,\beta$ of
$\cat{L}^p_{\struct{X}}(\DER[S]{\struct{X}}{\struct{X}},\,\cat{F})$, 
$\xi,\, \zeta$ of $\DER[S]{\struct{X}}{\struct{X}}$, $a$ of \struct{X}, and
$s$ of \struct{S}.

\item If $\alpha$ is alternating (respectively, symmetric), then so
is $\theta_\xi(\alpha)$, that is,
$$
\theta_\xi(\cat{A}lt^p_{\struct{X}}(\DER[S]{\struct{X}}{\struct{X}},\,
\cat{F}))
\,\subset\, \cat{A}lt^p_{\struct{X}}(\DER[S]{\struct{X}}{\struct{X}},\, \cat{F})
$$
(respectively, $\theta_\xi(\cat{S}ym^p_{\struct{X}}(\DER[S]{\struct{X}}{\struct{X}},\,
\cat{F}))
\,\subset\, \cat{S}ym^p_{\struct{X}}(\DER[S]{\struct{X}}{\struct{X}},\,\cat{F})$).
\end{enumerate}
\end{remark}

\subsubsection{The Lie derivative and the exterior product}\label{Sec:Lie-Der-Ext}

Let $\cat{F}$, $\cat{G}$, and $\cat{H}$ be $\struct{X}$-modules equipped with $S$-connections 
$D^\cat{F}$, $D^\cat{G}$ and $D^\cat{H}$ respectively. Let $\mu\,:\, \cat{F} \times \cat{G} \,
\longrightarrow\, \cat{H}$
be an $\struct{X}$-bilinear map. Take integers $p\,\geq\, 1$ and $q\,\geq\, 1$. Then, we have an
$\struct{X}$-bilinear map
$$
\wedge\,:\,\cat{A}lt^p_{\struct{X}}(\DER[S]{\struct{X}}{\struct{X}},\,\cat{F}) 
\times \cat{A}lt^q_{\struct{X}}(\DER[S]{\struct{X}}{\struct{X}},\,\cat{G}) 
$$
$$
\,\longrightarrow\,
\cat{A}lt^{p+q}_{\struct{X}}(\DER[S]{\struct{X}}{\struct{X}},\, \cat{H})\, .
$$

Suppose that $D^\cat{F}$, $D^\cat{G}$, $D^\cat{H}$ and $\mu$ are compatible, 
that is,
$$
D^\cat{H}_\xi(\mu(u,v)) \,=\, \mu(D^\cat{F}_\xi(u),v) + 
\mu(u, D^\cat{G}_\xi(v))
$$
for all local sections $\xi$ of $\DER[S]{\struct{X}}{\struct{X}}$, $u$ of
$\cat{F}$, and $v$ of \cat{G}.

Then we have
$$
\theta_\xi(\alpha \wedge \beta)\,=\, \theta_\xi(\alpha)\wedge\beta + 
\alpha \wedge \theta_\xi(\beta)
$$
where the Lie derivations are associated with their respective connections,
while
$\alpha$ and $\beta$ are local sections of their respective $\struct{X}$-modules.

\subsection{Covariant derivative}\label{Sec:Cov-Der} 

Let $\cat{F}$ be an $\struct{X}$-module and 
$\xi \in \DER[S]{\struct{X}}{\struct{X}}(U)$, 
where $U \,\subset\, X$ is an open subset, and $p \geq 2$ an integer. For
each $$\alpha\,\in\, \cat{L}^p_{\struct{X}}(\DER[S]{\struct{X}}{\struct{X}},\cat{F})(U)\, ,$$
define $(\imath_\xi(\alpha))_U \,\in\,
\cat{L}^{p-1}_{\struct{X}}
(\DER[S]{\struct{X}}{\struct{X}},\cat{F})(U)$ by
$$
(\imath_\xi(\alpha))_U(\eta_1,\cdots,\eta_{p-1})\, =\, \alpha(\xi,\eta_1,\cdots
,\eta_{p-1})
$$
for all $\eta_1,\,\cdots,\, \eta_{p-1} \,\in\, \DER[S]{\struct{X}}{\struct{X}}(U)$.
When $\alpha$ is of degree $0$, we define $\imath_\xi(\alpha)_U \,=\, 0$.
We call 
$(\imath_\xi(\alpha))_U$ the \textbf{relative interior product} of $\xi$ and $\alpha$ over 
$U$. This yields an $\struct{X}$-module homomorphism 
$$
\imath\,:\, \DER[S]{\struct{X}}{\struct{X}} \,\longrightarrow\, 
\HOM[\struct{X}]{\cat{L}^p_{\struct{X}}(\DER[S]{\struct{X}}{\struct{X}}
,\,\cat{F})}{\cat{L}^{p-1}_{\struct{X}}(\DER[S]{\struct{X}}{\struct{X}}
,\,\cat{F})}
$$
defined by $\imath_U(\xi)(\alpha)\, =\, \imath_\xi(\alpha)_U$, for every open subset $U$ of $X$.

The interior product satisfies the following properties:
\begin{enumerate}
\item $\imath_{\xi + \eta} \,=\, \imath_\xi + \imath_\eta$, for all
local sections $\xi$ and $\eta$ of $\DER[S]{\struct{X}}{\struct{X}}$.

\item $\imath_{a\xi} \,=\, a \imath_\xi $, for all local sections 
$a$ of $\struct{X}$ and $\xi$ of $\DER[S]{\struct{X}}{\struct{X}}$.

\item If $D$ is an $S$-connection on \cat{F}, and $\theta$ the
associated relative Lie derivative, then for all local sections $\xi,\, \eta$ of 
$\DER[S]{\struct{X}}{\struct{X}}$,
$$
\theta_\xi \circ \imath_\eta - \imath_\eta \circ \theta_\xi \,=\, 
\imath_{[\xi,\eta]}\, .
$$

\item If $\alpha$ is a local section of
$\cat{A}lt^p_{\struct{X}}(\DER[S]{\struct{X}}{\struct{X}}$, then 
$\imath_\xi(\imath_\xi(\alpha)) \,=\, 0$.

\item Let $\cat{F}$, $\cat{G}$ and $\cat{H}$ be \struct{X}-modules equipped with 
$S$-connections $D^\cat{F}$, $D^\cat{G}$ and $D^\cat{H}$ respectively. Let 
$\mu\,:\, \cat{F} \times \cat{G} \,\longrightarrow\, \cat{H}$ be an
$\struct{X}$-bilinear map. Let 
$p \,\geq\, 1$ and $q \,\geq\, 1$ be integers. Then, we have an $\struct{X}$-bilinear map
$$
\wedge\,:\,\cat{A}lt^p_{\struct{X}}(\DER[S]{\struct{X}}{\struct{X}},\,\cat{F}) 
\times \cat{A}lt^q_{\struct{X}}(\DER[S]{\struct{X}}{\struct{X}},\,\cat{G}) \\
\,\longrightarrow\,\cat{A}lt^{p+q}_{\struct{X}}(\DER[S]{\struct{X}}{\struct{X}},\,\cat{H})\, .
$$ 
Suppose that $D^\cat{F}$, $D^\cat{G}$, $D^\cat{H}$ and $\mu$ are compatible,
that is,
$$
D^\cat{H}_\xi(\mu(u,v)) \,=\, \mu(D^\cat{F}_\xi(u),v) + 
\mu(u, D^\cat{G}_\xi(v))
$$
for all local sections $\xi$ of $\DER[S]{\struct{X}}{\struct{X}}$, $u$ of
\cat{F}, and $v$ of \cat{G}. Then
$$
\imath_\xi(\alpha \wedge \beta) \,=\, \imath_\xi(\alpha)\wedge\beta +
(-1)^p \alpha \wedge \imath_\xi(\beta)\,,
$$
where the Lie derivations are associated with their respective connections,
while $\alpha$ and $\beta$ are local sections of their respective $\struct{X}$-modules.
\end{enumerate}

\begin{proposition}\label{pro:2}
Let $D$ be an $S$-connection on an $\struct{X}$-module \cat{F}. Then, there
exists a unique family of $S$-linear morphism 
$$
d \,=\, d_p\,:\,\cat{A}lt^p_{\struct{X}}(\DER[S]{\struct{X}}{\struct{X}},\,\cat{F})
\,\longrightarrow\,
\cat{A}lt^{p+1}_{\struct{X}}(\DER[S]{\struct{X}}{\struct{X}},\, \cat{F})\, ,
$$
where $p\, \in\, \mathbb N$, such that 
$$
\theta_\xi \,=\, d \circ \imath_\xi + \imath_\xi \circ d 
$$
for all local sections $\xi$ of $\DER[S]{\struct{X}}{\struct{X}}$.
\end{proposition}

\begin{proof}
This is proved in \cite[p.~11, Chapter \textbf{I}, Theorem 2]{K1}.
\end{proof}

The $S$-linear morphism $d$ in Proposition \ref{pro:2}
is called the \textbf{covariant derivative} associated with $D$.

\subsubsection{Explicit formula for the covariant derivative}\label{Exp-for-cov-der}

The following proposition is straight-forward.

\begin{proposition}\label{lem:1}
Let $D$ be an $S$-connection on an $\struct{X}$-module $\cat{F}$. Then, the
covariant derivative with respect to $D$ is given by 
$$
d(\alpha)(\xi_1,\cdots ,\xi_{p+1}) \,=\, \sum^{p+1}_{i = 1} (-1)^{i+1} D_{\xi_i
}(\alpha(\xi_1,\cdots ,\hat{\xi_i},\ldots,\xi_{p+1})) 
$$
$$
+ \,
\sum_{1 \leq i < j \leq p+1} (-1)^{(i+1)} \alpha([\xi_i,\xi_j],\xi_1,\ldots
,\hat{\xi_i},\ldots,\hat{\xi_j},\ldots,\xi_{p+1})
$$
for all local sections $\alpha$ of 
$\cat{A}lt^p_{\struct{X}}(\DER[S]{\struct{X}}{\struct{X}},\cat{F})$ and 
$\xi_1,\,\cdots,\,\xi_{p+1}$ of $\DER[S]{\struct{X}}{\struct{X}}$.
\end{proposition}

Proposition \ref{lem:1} gives the following:

\begin{corollary}\label{cor:2}
The following three hold.
\begin{enumerate}
\item $ d(\alpha)(\xi) \,=\, D_\xi(\alpha)$
for all local sections $\alpha $ of \cat{F} and $\xi$ of 
$\DER[S]{\struct{X}}{\struct{X}}$.

\item \label{27}
$d(\alpha)(\xi,\eta) \,=\, D_\xi(\alpha(\eta)) - D_\eta(\alpha(\xi)) - 
\alpha([\xi,\eta])$,
for all local sections $\alpha$ of
$\cat{A}lt^2_{\struct{X}}(\DER[S]{\struct{X}}{\struct{X}},\,\cat{F})$ and 
$\xi,\,\eta$ of $\DER[S]{\struct{X}}{\struct{X}}$.

\item
$d(\alpha)(\xi,\eta,\nu) \,=\, \sum_{cyclic} (D_\xi(\alpha(\eta,\nu)) 
- \alpha([\xi,\eta],\nu))$.
\end{enumerate}
\end{corollary}

\subsubsection{Covariant derivative and exterior product}\label{Cov-der-ext-pro}

The following proposition is straight-forward.

\begin{proposition}\label{pro:3}
Let $\cat{F}$, $\cat{G}$ and $\cat{H}$ be \struct{X}-modules equipped with $S$
connections $D^\cat{F}$, $D^\cat{G}$ and $D^\cat{H}$ respectively. Let $ \mu\,:\, 
\cat{F} \times \cat{G}
\,\longrightarrow\, \cat{H}$ be an $\struct{X}$-bilinear map. Suppose that $ D^\cat{F}$,
$\cat{G}$, $D^\cat{H}$ and $\mu$ are compatible. Then for all local sections $\alpha$ of
$\cat{A}lt^p_{\struct{X}}(\DER[S]{\struct{X}}{\struct{X}},\,\cat{F})$ and 
$\beta$ of $\cat{A}lt^\bullet_{\struct{X}}(\DER[S]{\struct{X}}{\struct{X}},\,\cat{F})$,
we have
$$
d(\alpha \wedge \beta) \,=\, d(\alpha) \wedge \beta + 
 (-1)^p \alpha \wedge d(\beta)\, .
$$
\end{proposition}

Proposition \ref{pro:3} gives the following:

\begin{corollary}\label{cor:3}
Let \cat{F} be an \struct{X}-module with connection $D$. Then
$$
d(a \alpha) \,=\, d(a) \wedge \alpha + a d(\alpha)
$$
for all local sections $a$ of $\struct{X}$ and $\alpha$ of
$\cat{A}lt^\bullet_{\struct{X}}(\DER[S]{\struct{X}}{\struct{X}},\,\cat{F})$,
where $d(a)$ is the covariant derivative of $a$ with respect to the
canonical connection on $\struct{X}$.
\end{corollary}

\subsection{The curvature form}\label{Cur-Form}

Let $D$ be an $S$-connection on an \struct{X}-module \cat{F}, and let
$$
d:\cat{A}lt^\bullet_{\struct{X}}(\DER[S]{\struct{X}}{\struct{X}},\cat{F})
\, \longrightarrow\, \cat{A}lt^{\bullet+1}_{\struct{X}}(\DER[S]{\struct{X}}{\struct{X}},\cat{F})
$$
be the covariant derivative associated with $D$. Then the map
$$
d \circ d \,=\, d^2\,:\,\cat{A}lt^\bullet_{\struct{X}}(\DER[S]{\struct{X}
}{\struct{X}},\,\cat{F}) \, \longrightarrow\, \cat{A}lt^{\bullet+2}_{\struct{X}}(\DER[S
]{\struct{X}}{\struct{X}},\,\cat{F})
$$
is called the curvature operator of $D$, and it will be denoted by $R$.

Let $ \alpha $ be a local section of $\cat{A}lt^0_{\struct{X}}(\DER[S
]{\struct{X}}{\struct{X}},\,\cat{F})\,=\,\cat{F}$. 
Then $R(\alpha)\,=\,d(d(\alpha))$ is a local section of 
$\cat{A}lt^2_{\struct{X}}(\DER[S]{\struct{X}}{\struct{X}},\,\cat{F})$.
Let $\xi$ and $\eta$ be local sections of $\DER[S]{\struct{X}}{\struct{X}}$. Then
\begin{align*}
 R(\alpha)(\xi, \eta) & \,=\, d(d(\alpha))(\xi,\,\eta) \\
& \,=\, D_\xi(d(\alpha)(\eta)) - D_\eta(d(\alpha)(\xi))- d(\alpha)([\xi,\,\eta])\\
& \,=\, D_\xi(D_\eta(\alpha)) - D_\eta(D_\xi(\alpha)) - D_{[\xi, \eta]}(\alpha)\, .
\end{align*} 

Thus, for every open subset $U$ of $X$ and for all sections
$\xi,\,\eta\,\in\, \DER[S]{\struct{X}}{\struct{X}}(U)$, we get an 
$\struct{X}|_U$-module homomorphism
$$
K_U(\xi,\eta)\,:\, \cat{F}|_U \, \longrightarrow\, \cat{F}|_U
$$
defined by
$$
K_U(\xi,\,\eta)\,= \,D_\xi \circ D_\eta - D_\eta \circ D_\xi - D_{[\xi,\eta]}\,.
$$
Hence these $K_U$ together define an $\struct{X}$-bilinear map
$$
K\,:\, \DER[S]{\struct{X}}{\struct{X}} \times 
\DER[S]{\struct{X}}{\struct{X}} \, \longrightarrow\, \END[\struct{X}]{\cat{F}}\, .
$$

The following is straight-forward.

\begin{proposition}\label{lem:2}
The above $\struct{X}$-bilinear map $K$ is an alternating map.
\end{proposition}

The alternating \struct{X}-bilinear map $K$ is called the \textbf{curvature
form} of $D$. We say the $S$-connection is \textbf{flat} if the curvature
form is identically zero.

\subsection{Connection and curvature matrices}\label{cc-mt}

Let $\cat{F}$ be a locally free coherent $\struct{X}$-module of
rank $r$. Let $U$ be an open subset of $X$ such that $\cat{F}|_U$ is a 
free $\struct{X}|_U$-module. Let $s \,=\, (s_1,\,\cdots,\,s_r)$ be an 
$\struct{X}|_U$-basis of $\cat{F}|_U$. For each 
$\xi\,\in\,\DER[S]{\struct{X}}{\struct{X}}(U)$, define an $r \times r$ matrix
$$
\omega(\xi) \,=\, (\omega_{ij}(\xi))_{1 \leq i,j \leq r}
$$
of elements of $\struct{X}(U)$ by the equation
$$
D_\xi(s_j) \,=\, \sum^r_{i = 1}\omega_{ij}(\xi)s_i \ \ (1 \,\leq\, j\, \leq\, r)\, .
$$
We, thus get, for every $i,\,j\, \in\, \{1,\,\cdots,\,r\}$, an element $\omega_{ij}$ of
$$\HOM[\struct{X}]{\DER[S]{\struct{X}}{\struct{X}}}{\struct{X}}(U) \,=\, 
 \cat{A}lt^1_{\struct{X}} (\DER[S]{\struct{X}}{\struct{X}},\struct{X})
 (U)\, .$$
 This gives an $r \times r$ matrix $\omega \,=\, (\omega_{ij})_{1 \leq i,j \leq
 r}$, where entries are sections of 
 $\cat{A}lt^1_{\struct{X}} (\DER[S]{\struct{X}}{\struct{X}},\struct{X}) $
 over $U$. It is called the 
 \textbf{connection matrix} of $D$ with respect to $s$. Considering 
 $s \,=\, (s_1,\,\cdots,\,s_r)$ as a row vector, this $\omega$ is the unique 
 $r \times r$ matrix over 
 $\cat{A}lt^1_{\struct{X}} (\DER[S]{\struct{X}}{\struct{X}},\struct{X})(U)$
 such that
$$
 D_\xi(\omega) \,=\, s \omega(\xi)
$$
for all $\xi \,\in\, \DER[S]{\struct{X}}{\struct{X}}(U)$. 
If $u\,=\, \sum^r_{j = 1} a_j s_j \,\in\, \cat{F}(U)$, where $a_{j}\,\in\, \struct{X}(U)$, then 
$$
 D_\xi(u)\,=\, s (\xi(a) + \omega(\xi)a)
$$
 for all $\xi \,\in\, \DER[S]{\struct{X}}{\struct{X}}(U)$.
 If
 $$d \,=\, d_{X/S}\,:\,\cat{A}lt^0_{\struct{X}} (\DER[S]{\struct{X}}{\struct{X}},\,\struct{X})
 \,=\, \struct{X}
\, \longrightarrow\, \cat{A}lt^1_{\struct{X}} (\DER[S]{\struct{X}}{\struct{X}},\,\struct{X}) 
 $$
 is the covariant derivative associated with canonical connection 
 on \struct{X}, then
$$
D_\xi(u) \,=\, s (d(a) + \omega(\xi)a)
$$
for all $\xi \,\in\, \DER[S]{\struct{X}}{\struct{X}}(U)$.

Let $t \,=\, (t_1,\,\cdots,\,t_r)$ be another $\struct{X}|_U$-basis of 
$\cat{F}|_U$, and $ t_j \,=\, \sum^r_{i=1} a_{ij}s_i$, for all
$1 \,\leq\, j \,\leq\, r$. Then the matrix
$a \,=\, (a_{ij})_{1 \leq i,j \leq r}$ is an element of
${\rm GL}_r(\struct{X}(U))$. Let $\omega'$ be the connection matrix of $D$ 
 with respect to $t$. Then we have
$$
 \omega' \,=\, a^{-1}da + a^{-1}\omega a\, .
$$
Let $K$ be the curvature form of $D$. For all 
$\xi,\, \eta \,\,in\, \DER[S]{\struct{X}}{\struct{X}}(U)$, 
let $$\Omega(\xi,\eta) \,=\, (\Omega_{ij}(\xi,\,\eta))_{1 \leq i, j \leq r}$$
 be the $ r\times r$ matrix over $\struct{X}(U)$, defined by
$$
 K(\xi,\,\eta)(s_j) \,=\, \sum^r_{ i = 1} \Omega_{ij}(\xi,\eta)s_i
$$
for $1 \,\leq\, j \,\leq\, r$.
We thus get for all $ i,\,j \,\in \,\{1,\,\cdots,\,r\} $ an element 
$$\Omega_{ij}\, \in\, 
\cat{A}lt^2_{\struct{X}}\big(\DER[S]{\struct{X}} \,\struct{X},\,
\struct{X}\big)(U)\, .$$ This gives a $r\times r$ matrix 
$\Omega\,=\, (\Omega_{ij})_{1 \leq i,j \leq r}$ whose entries are sections of
$\cat{A}lt^2_{\struct{X}} (\DER[S]{\struct{X}}{\struct{X}},\,\struct{X})$ over
$U$. It is called the \textbf{curvature matrix} of $D$ with respect to $s$. 
Considering $s\,=\, (s_1,\,\cdots,\,s_r)$ as a row vector, this $\Omega$ is the unique
$r\times r$ matrix over 
$\cat{A}lt^2_{\struct{X}} (\DER[S]{\struct{X}}{\struct{X}},\,\struct{X})(U)$
such that
$$
 K(\xi,\,\eta)s \,=\, s \Omega(\xi,\, \eta)
$$
 for all $\xi,\, \eta \,\in\, \DER[S]{\struct{X}}{\struct{X}}(U)$.
 
We have $\Omega \,=\, d \omega + \omega \wedge \omega.$
If $t \,=\, (t_1,\,\cdots,\,t_r)$ is another $\struct{X}|_U$-basis of 
$\cat{F}|_U$ as above, and $\Omega'$ is the curvature matrix of 
$D$ with respect to $t$, then 
$$
 \Omega' \,=\, a^{-1} \Omega a\, ,
$$
where $a \,=\, (a_{ij})_{1 \leq i,j \leq r}$ as before. 
 
\section{S-Differential Operator}
 
\subsection{First order S-differential operator}\label{First-diff}

Let $(\pi,\,\pi^\sharp)\,:\,\ringed{X} \,\longrightarrow\, \ringed{S}$ be a morphism of ringed
spaces. Let $\cat{F}$ and $\cat{G}$ be two $\struct{X}$-modules.
A first order $S$-differential operator is a morphism
$$
P \,:\, \cat{F} \,\longrightarrow\, \cat{G}
$$
of sheaves of abelian groups such that 
\begin{enumerate}
\item $P$ is an $S$-linear morphism, and

\item for every open subset $U\, \subset\, X$ and for every 
$f \,\in\, \struct{X}(U)$, the bracket 
$[P|_U,\,f]\,:\,\cat{F}|_U \,\longrightarrow\, \cat{G}|_U$ defined as 
$$
[P|_U,\, f]_V(s) \,=\, P_V(f|_V s) - f|_V P_V(s)
$$
is an $\struct{U}$-module homomorphism for every open subset $V\, \subset\, 
U$ and all $ s \,\in \,\cat{F}(V)$.
\end{enumerate}
Let $\Diff[1]{S}{\cat{F}}{\cat{G}}$ denote the set of all
first order $S$-differential operator. Then 
$\Diff[1]{S}{\cat{F}}{\cat{G}}$ is an
$\struct{X}(X)$-module. For every open subset $U$ of $X$, 
$U\,\longmapsto\, \Diff[1]{S}{\cat{F}|_U}{\cat{G}|_U}$ is a sheaf of first order
$S$-differential operator from $\cat{F}|_U$ to 
$\cat{G}|_U$. This sheaf is denoted by $\DIFF[1]{S}{\cat{F}}{\cat{G}}$.

\subsection{Symbol of a first order S-differential operator}\label{Symb-op}

Given a first order $S$-differential operator $P\,:\, \cat{F} \,\longrightarrow\, \cat{G}$,
define a morphism of abelian sheaves 
$$
\theta\,:\, \struct{X} \,\longrightarrow\, \HOM[\struct{X}]{\cat{F}}{\cat{G}}
$$
by $\theta_U(f) \,=\, [P|_U,\,f]$
for every open subset $U\, \subset\, X$ and $ f \,\in \,\struct{X}(U)$.
Then $\theta$ is an $S$-derivation. For every $W \,\subset\, S$, 
$V \,\subset\, \pi^{-1}(W)$, $s\,\in\, \struct{S}(W)$, $t \,\in\, \struct{X}(V)$, 
and $u \in \cat{F}(V)$, we have, by $S$-linearity of $P$,
\begin{align*}
\theta_V(s|_V t)(u) & \,=\, [P|_V,\,s|_V t](u)\\
& \,=\, P(s|_V t)(u) - s|_V t P(u) \\
& \,=\, (s|_V) P(t)u - s|_V t P(u) \\
& \,=\, (s|_V)[P|_V,t](u)\, .
\end{align*}
Thus $\theta$ is an $S$-linear morphism.

Next, to verify $\theta$ satisfies Leibniz rule, for every
$f ,\, g \,\in\, \struct{X}(U)$, where $U\, \subset\,X$ is any
open subset, and for every $u \,\in\, \cat{F}(U)$, we have
$[P|_U,\, f](g u) \,= \, P(f g u) - f P(g u)$
which gives $P(fgu) \,=\, g [P,f](u) + fP(gu)$.
Now, 
\begin{align*}
\theta(fg)(u) & \,=\, [P,\,fg](u)\\
& \,=\, P(fgu) - fgP(u)\\
& \,=\, g[P,f](u) + fP(gu) - fgP(u)\\
& \,=\, g \theta(f)(u) + f \theta(g)(u)\\
& \,=\, (g \theta(f) + f \theta(g))(u) .\\
\end{align*}
Thus, $ \theta(fg) \,=\, \theta(f)g + f \theta(g)$.

Hence, we have the following:

\begin{proposition}\label{prop}
Let $\cat{F}$ and $\cat{G}$ be $\struct{X}$-modules and
$P\,:\, \cat{F} \,\longrightarrow\, \cat{G}$ a first order $S$-differential operator.
Then there exists a unique $S$-derivation
$$
\theta\,:\, \struct{X} \,\longrightarrow\, \HOM[\struct{X}]{\cat{F}}{\cat{G}}
$$
such that $\theta_U(f) \,=\, [P|_U,\,f]$
for every open subset $U$ of $X$ and for every $f \,\in\, \struct{X}(U)$.
\end{proposition}

The $S$-derivation $\theta$ defined above is called the 
\textbf{symbol of $P$}; it will be denoted by $\sigma_1(P)$.

Every $\struct{X}$-module homomorphism is a first order $S$-differential
operator. Therefore, \HOM[\struct{X}]{\cat{F}}{\cat{G}} is an 
\struct{X}-submodule of $\DIFF[1]{S}{\cat{F}}{\cat{G}}$. Let
$$ 
\imath \,:\, \HOM[\struct{X}]{\cat{F}}{\cat{G}}
\,\longrightarrow\,\DIFF[1]{S}{\cat{F}}{\cat{G}} 
$$ 
be the inclusion morphism. Thus, we have an exact sequence of \struct{X}-modules
\begin{equation}\label{symb}
0 \,\longrightarrow\, \HOM[\struct{X}]{\cat{F}}{\cat{G}}\,
\stackrel{\imath}{\longrightarrow}\,
\DIFF[1]{S}{\cat{F}}{\cat{G}} 
\,\stackrel{\sigma_1}{\longrightarrow}\, \DER[S]{\struct{X}}{\HOM[\struct{X}]
{\cat{F}}{\cat{G}}} \, .
\end{equation}
In general, $\sigma_1$ need not be surjective.

\section{Analytic Theory}\label{sec:An-th}
 
\subsection{Complex analytic families}\label{com-ana-fam}

Let $(S,\struct{S})$ be a complex manifold of dimension $n$. For each 
$t\,\in \,S$, let there be given a compact connected complex manifold $X_t$ 
of fixed dimension $l$. The set $\{X_t\,\mid\, t \in S \}$ of compact
connected complex manifolds is called a \emph{complex analytic family of
compact connected complex manifolds}, if there is a complex manifold 
$(X,\, \struct{X})$ and a surjective holomorphic map $\pi\,:\, X\,\longrightarrow\,
S$ of complex manifolds with connected fibers such that the followings hold:
\begin{enumerate}
\item $\pi^{-1}(t) \,=\, X_t$ for all $t\,\in\, S$,

\item $\pi^{-1}(t)$ is a compact connected complex 
submanifold of $X$ for all $t \,\in\, S$, and

\item the rank of the Jacobian matrix of $\pi$ is equal to 
$n$ at each point of $X$. 
\end{enumerate}
(See \cite{KS1} for details.)
In other words, $\pi \,:\, X \,\longrightarrow\,S$ is a surjective
holomorphic proper submersion, such that $\pi^{-1}(t) \,=\, X_t$ is connected for every $t
\,\in\, S$. 
 
\subsection{Holomorphic relative tangent and cotangent bundles}\label{Rel-tan-cotan-bun}

Let $\pi\,:\, X\,\longrightarrow\, S$ be a surjective holomorphic submersion of complex
manifolds with connected fibers such that $\dim X\,=\, m$ and $\dim S\,=\, n$. For any $t\, \in\, t$, the fiber
$\pi^{-1}(t)$ will be denoted by $X_t$.
Let $d\pi_S\,:\, TX \,\longrightarrow\, \pi^*TS$ be the differential of $\pi$.
The subbundle $$T(X/S)\, :=\, \Ker{d\pi_S}\, \subset\, TX$$ is called the relative tangent
bundle for $\pi$. Thus we have a short exact sequence
\begin{equation}\label{eq:ses1}
0 \,\longrightarrow\, \cat{T}_{X/S}\, \stackrel{\imath}{\longrightarrow}\,
\cat{T}_X \, \stackrel{d\pi_S}{\longrightarrow}\, \pi^* \cat{T}_S \,\longrightarrow\, 0
\end{equation}
of $\struct{X}$-modules and $\struct{X}$-linear maps.

The dual $T(X/S)^*$ of the relative tangent bundle is
called the relative cotangent bundle and it is denoted by $\Omega^1(X/S)$.
The sheaf of holomorphic sections of relative cotangent bundle 
$\Omega^1(X/S)$ will also be denoted by $\Omega^1_{X/S}$. Dualizing the short
exact sequence in \eqref{eq:ses1}, we get a short exact sequence
\begin{equation}\label{eq:ses2}
0 \,\longrightarrow\, \pi^* \Omega^1_{S} 
\, \stackrel{d\pi^*_S}{\longrightarrow}\,
\Omega^1_{X} \, \stackrel{\imath^*}{\longrightarrow}\,
\Omega^1_{X/S} \,\longrightarrow\, 0\, .
\end{equation}
The relative tangent sheaf $\cat{T}_{X/S}$ and the relative cotangent 
sheaf $\Omega^1_{X/S}$ are locally free $\struct{X}$-modules of rank
$l\,=\,m-n$.

From Theorem \ref{thm:1}, there exists a unique $S$-derivation 
$d_{X/S}\,:\,\struct{X} \,\longrightarrow\, \Omega^1_{X/S}$. For any integer
$r\,\geq\, 1$, define $\Omega^r_{X/S} \,=\, \Lambda^r \Omega^1_{X/Y}$, which
is called the sheaf of holomorphic relative $r$-forms on $X$ over
$S$. We have the short exact sequence
\begin{equation}\label{eq:ses3}
0 \,\longrightarrow\, \pi^*\Omega^1_S \otimes \Omega^{r-1}_X\,\longrightarrow\,
\Omega^r_X \,\longrightarrow\, \Omega^r_{X/S} \,\longrightarrow\, 0
\end{equation}
which is derived from the short exact sequence in \eqref{eq:ses2}.

\begin{theorem}\label{thm:1.5}
There exists canonical $\pi^{-1}\struct{S}$-linear maps
$\partial^r_{X/S}\,:\,\Omega^r_{X/S} \,\longrightarrow\,\Omega^{r+1}_{X/S}$, called the 
 relative exterior derivative, satisfying the following:
\begin{enumerate}
\item \label{1.51} $\partial^0_{X/S}\,= \,d_{X/S}\,:\, \struct{X}
\,\longrightarrow\,\Omega^1_{X/S}$,

\item \label{1.52} $\partial^{r+1}_{X/S} \circ \partial^r_{X/S} \,=\, 0$, and
 
\item \label{1.53} $\partial^{r+s}_{X/S}(\alpha \wedge \beta)
\,=\, \partial^r_{X/S} \alpha \wedge \beta + (-1)^r \alpha \wedge 
\partial^s_{X/S} \beta$, for all local sections $\alpha$ of 
 $\Omega^r_{X/S}$ and $\beta$ of $\Omega^s_{X/S}$.
\end{enumerate}
\end{theorem}
 
\begin{proof}
This basically follows from Proposition \ref{pro:2}. First note that the sheaf 
$\DER[S]{\struct{X}}{\struct{X}}$ is canonically isomorphic to the holomorphic relative tangent 
sheaf $\cat{T}_{X/S}$ as an $\struct{X}$-module, and hence 
$\cat{A}lt^r_{\struct{X}}(\DER[S]{\struct{X}}{\struct{X}},\, \struct{X})$ is canonically 
isomorphic to the sheaf $\Omega^r_{X/S}$ as an $\struct{X}$-module. In view of the canonical 
$S$-connection in $\struct{X}$, by Proposition \ref{pro:2}, canonical $S$-linear map
exists satisfying \eqref{1.51}, and by Proposition \ref{pro:3}, it satisfies \eqref{1.53}. 
Finally, \eqref{1.52} follows from Corollary \ref{cor:2}\eqref{27}.
\end{proof}

\subsection{Relative Atiyah algebra}\label{Atiyah-alg}
 
\begin{proposition}[{Symbol exact sequence}]\label{pro:4}
Let $\pi \,:\, X \,\longrightarrow\, S$ be a holomorphic proper submersion of complex manifolds
with connected fibers,
and let $\cat{F}$ and $\cat{G}$ be two locally free $\struct{X}$-modules of rank $r$ and $p$
respectively. Then
$$
0 \,\longrightarrow\, \HOM[\struct{X}]{\cat{F}}{\cat{G}}\,\stackrel{\imath}{\longrightarrow}\,
\DIFF[1]{S}{\cat{F}}{\cat{G}} \,\stackrel{\sigma_1}{\longrightarrow}\,
\DER[S]{\struct{X}}{\HOM[\struct{X}]{\cat{F}}{\cat{G}}}\,\longrightarrow\, 0
$$
is an exact sequence of $\struct{X}$-modules.
\end{proposition}
 
\begin{proof}
It is enough to show that $\sigma_1$ is surjective. Let
$\theta \,\in\, (\DER[S]{\struct{X}}{\HOM[\struct{X}]{\cat{F}}{\cat{G}}})_x$ with
$x\, \in\, X$. We have to show that there exists a first order $S$-differential operator 
$P$ defined near $x$, such that $(\sigma_1)_x(P_x)\,=\, \theta$, where
$$P_x\,\in\, (\DIFF[1]{S}{\cat{F}}{\cat{G}})_x$$ is the germ of $P$ at $x$. Let
$(U,\, \phi \,=\, (z_1,\,\cdots,\,z_l,\,z_{l+1},\,\cdots,\,z_{l+n}))$ be a holomorphic chart
on $X$ around $x$, and let $s \,=\, (s_1,\,\cdots,\,s_l)$ and $t \,=\, (t_1,\,\cdots,\,t_p)$
be holomorphic frames of $\cat{F}$ and $\cat{G}$ respectively, on $U$. We may
assume that $\theta$ is the germ at $x$ of a section $u$ of
$\DER[S]{\struct{X}}{\HOM[\struct{X}]{\cat{F}}{\cat{G}}}$ over $U$. Since
$$\HOM[\struct{X}]{\Omega^1_{X/S}}{\HOM[\struct{X}]{\cat{F}}{\cat{G}}}
\,\cong\, \DER[S]{\struct{X}}{\HOM[\struct{X}]{\cat{F}}{\cat{G}}}\, ,$$
$u$ can be considered as a section of
$\HOM[\struct{X}]{\Omega^1_{X/S}}{\HOM[\struct{X}]{\cat{F}}{\cat{G}}}$ over
$U$, that is,
$u\,:\, \Omega^1_{X/S}|_U \,\longrightarrow\, \HOM[\struct{U}]{\cat{F}|_U}{\cat{G}|_U}$ is an
$\struct{U}$-module homomorphism.
As $\{dz_\alpha\,\mid\, 1\,\leq\, \alpha \,\leq\, l\}$ is an $\struct{U}$-basis of
$\Omega^1_{X/S}|_U$, there exists a uniquely determined function $b^\alpha_{ij}\,\in\,
\struct{X}(U)$, where $1 \,\leq\, i \,\leq\, p$, $1 \,\leq\, j \,\leq\, r$ and
$1\,\leq\, \alpha \,\leq\, l$, such that 
$$
u(dz_\alpha)(s_j) \,= \,\sum^p_{i = 1} b^\alpha_{ij} t_i\, .
$$
Define $P \,:\, \cat{F}|_U \,\longrightarrow\, \cat{G}|_U$ by
$$
P(\sum^r_{j = 1}(f_j s_j) \,=\,
\sum_{i,j,\alpha} b^\alpha_{ij} {\frac{\partial f_j}{\partial z_\alpha}} t_i\, .
$$
Then $P$ is $S$-linear, because for any $V\,\subset\, \pi^{-1}(W) \bigcap U$,
where $W\, \subset\, S$ is any open subset, and any $g \,\in\, \struct{S}(W)$, we have
$$\frac{\partial g \circ \pi}{\partial z_\alpha} \,=\,
d\pi_x(\frac{\partial}{\partial z_\alpha})(g) \,=\, 0$$ for all $\alpha\,=\, 1,\,\cdots,\,l$.
The bracket operation $[P,\, f]$ is $\struct{U}$-linear. Thus $P$ is a first order $S$-differential
operator, that is, $P \,\in\, \DIFF[1]{S}{\cat{F}}{\cat{G}}(U)$.
Let $V\,\subset\, U$ and $\xi\,\in\, \Omega^1_{X/S}(V)$. Then $\xi
\,= \,\sum^l_{\alpha = 1} \xi_\alpha dz_\alpha$, where $\xi_\alpha \,\in\, \struct{U}(V)$. Then
by the construction of the symbol map, and the universal property of $(\Omega^1_{X/S},\, d_{X/S})$,
we have
\begin{equation}\label{eq:76}
\sigma_1(P)_V(\xi) \,=\, \sum_\alpha \xi_\alpha \sigma_1(P)_V (dz_\alpha)\,
= \,\sum_\alpha \xi_\alpha [P, z_\alpha]\, .
\end{equation}
{}From the definition of $P$ it follows that $P(s_j) \,=\, 0$ for all $j$, and hence
$$
[P,\, z_\alpha](s_j)\,=\, P(z_\alpha s_j)\,=\, \sum_i b^\alpha_{ij} t_i \,=\, u_V(dz_\alpha)(s_j)\, .
$$
Therefore, $[P,\, z_\alpha] \,=\, u(dz_\alpha)$, and hence \eqref{eq:76} becomes
$$
\sigma_1(P)_V(\xi)\, =\, \sum_\alpha \xi_\alpha u_V (dz_\alpha) \,=\, u_V(\xi)\,.
$$
This proves that $\sigma_1(P)_U \,=\, u $, that is, $ (\sigma_1)_x(P_x)\,=\,\theta$.
\end{proof}

Let $E$ be a locally free $\struct{X}$-module. By Proposition \ref{pro:4}, we have a
short exact sequence of \struct{X}-modules
$$
0\, \longrightarrow\, \END[\struct{X}]{E}\,\stackrel{\imath}{\longrightarrow}\, 
\DIFF[1]{S}{E}{E}\,\stackrel{\sigma_1}{\longrightarrow}\,
\DER[S]{\struct{X}}{\END[\struct{X}]{E}} \,\longrightarrow\, 0\, .
$$
For any $S$-derivation $\xi \,:\, \struct{X} \,\longrightarrow\, 
\struct{X}$, let $\widetilde{\xi}\,:\, \struct{X} \,\longrightarrow\, 
\END[\struct{X}]{E}$ be the map defined by $a\,\longmapsto\, \xi(a) \id{E}$, where
$a$ is a local sections of $\struct{X}$. Then $\widetilde{\xi}$ is an $S$-derivation.
Thus, we have an \struct{X}-module homomorphism
$$
\Psi \,:\, \DER[S]{\struct{X}}{\struct{X}} \,\longrightarrow\,
\DER[S]{\struct{X}}{\END[\struct{X}]{E}}
$$
defined by $\xi\,\longmapsto \,\widetilde{\xi}$. Note that $\Psi$ is an injective homomorphism.

Define 
$$
\cat{A}t_S(E)\,=\, \sigma^{-1}_1(\Psi(\DER[S]{\struct{X}}{\struct{X}}))\, ,
$$
which is an \struct{X}-module and for every open subset $U$ of $X$. Note that
$\cat{A}t_S(E)(U)$ consists of first order $S$-differential operator $P\,\in\,
\DIFF[1]{S}{E}{E}(U)$ such that $(\sigma_1)_{U}(P)\,=\, \Psi(\xi)$ for some
$\xi \,\in\, \DER[S]{\struct{X}}{\struct{X}}(U)$, which is equivalent to the assertion that
$ \sigma_1(P)(a)\,=\, \xi(a) \id{E}$ or 
$[P,\, a] \,=\, \xi(a) \id{E}$, for all $a\,\in\, \struct{X}(U)$ and for some $\xi \,\in\,
\DER[S]{\struct{X}}{\struct{X}}(U)$. Hence, we get a short exact sequence
\begin{equation}\label{eq:82}
0\,\longrightarrow\, \END[\struct{X}]{E}\,\stackrel{\imath}{\longrightarrow}\, \cat{A}t_S(E)
\,\stackrel{\sigma_1}{\longrightarrow}\, 
\DER[S]{\struct{X}}{\struct{X}} \,\longrightarrow\, 0
\end{equation}
of \struct{X}-modules, which is called the \textbf{Atiyah sequence} while
$\cat{A}t_S(E)$ is called the \textbf{relative Atiyah algebra} of $E$.

\begin{proposition}\label{pro:5}
Let $\pi \,:\, X \,\longrightarrow\, S$ be a holomorphic proper submersion of complex 
manifolds with connected fibers, and let $E$ be a holomorphic vector bundle over $X$. Then $E$ admits 
an holomorphic $S$-connection if and only if the Atiyah sequence 
in \eqref{eq:82} splits holomorphically.
\end{proposition}
 
\begin{proof}
Suppose that the Atiyah sequence in \eqref{eq:82} splits holomorphically, that is, there exists an
$\struct{X}$-module homomorphism 
$$
\nabla \,:\, \DER[S]{\struct{X}}{\struct{X}} \,\longrightarrow\,\cat{A}t_S(E)
$$
such that $\sigma_1 \circ \nabla\,=\,\id{\DER[S]{\struct{X}}{\struct{X}}}$.
Then, for every open subset $U\,\subset\, X$ and for every 
$\xi\,\in\, \DER[S]{\struct{X}}{\struct{X}}(U)$, this $\nabla_U(\xi)$ 
is a first order $S$-differential operator such that $\sigma_1(\nabla_U(\xi))(a)
\,=\, \xi'(a) \id{E}$ for some $\xi' \,\in\, \DER[S]{\struct{X}}{\struct{X}}(U)$
and for every $a \,\in\, \struct{X}(U)$. This implies that $\xi \,=\, \xi'$,
because the Atiyah sequence splits. We have
$[\nabla_U(\xi),\,a] \,= \,\xi(a) \id{E}$, which can be expressed as 
$$
\nabla_U(\xi)(as) \,=\, a \nabla_U(\xi)(s) + \xi(a)s$$ for every 
$s\,\in\, E(U)$. Thus $\nabla_U(\xi)$ satisfies Leibniz rule, and since
$\cat{A}t_S(E)$ is an \struct{X}-submodule of 
$\END[S]{\cat{E}}$, it follows that $\nabla$ is actually an $S$-connection on $E$.

The converse follows from the fact that any $S$-connection satisfies Leibniz rule,
because it gives an splitting of Atiyah exact sequence. 
\end{proof}

The extension class of the Atiyah exact sequence \eqref{eq:82} of a holomorphic 
vector bundle $E$ over $X$ is an element of 
$\coh{1}{X}{\HOM[\struct{X}]{\cat{T}_{X/S}}{\END[\struct{X}]{E}}}$. This extension
class is called the relative Atiyah class of $E$, and it is denoted by at$_S(E)$. 
Note that
$$
\coh{1}{X}{\HOM[\struct{X}]{\cat{T}_{X/S}}{\END[\struct{X}]{E}}}\,=\,
\coh{1}{X}{\Omega^1_{X/S}(\END[\struct{X}]{E})}\, .
$$

Proposition \ref{pro:5} has the following corollary:

\begin{corollary}\label{cor:4}
A holomorphic vector bundle $E$ on $E$ admits a holomorphic $S$-connection if and only 
if its relative Atiyah class $\at[S]{E}\, \in\,
\coh{1}{X}{\Omega^1_{X/S}(\END[\struct{X}]{E})}$ is zero.
\end{corollary}

\subsection{Induced family of holomorphic connections}\label{Ind}

As before, $\pi\,:\,X \,\longrightarrow\, S$ is a surjective holomorphic proper submersion
with connected fibers.
Let $\varpi\,:\,E \,\longrightarrow\, X$ be a holomorphic vector bundle. For every $t \,\in\, S$,
the restriction of $E$ to $X_t \,=\, \pi^{-1}(t)$ is denoted by $E_t$.
Let $U$ be an open subset of $X$ and $s\, :\,U \,\longrightarrow\, E$ a holomorphic section.
We denote by $r_t(s)$ the restriction of $s$ to $X_t\cap U$, whenever $U \cap
X_t\,\neq\, \emptyset$. Clearly, $r_t(s)$ is a holomorphic section of $E_t$ over
$U \cap X_t$. The map $r_t \,:\, s\,\longmapsto\, r_t(s)$ induces, therefore, a 
homomorphism of $\C$-vector spaces from $E$ to $E_t$, which is denoted
by the same symbol $r_t$ (the restriction map $r_t$ is discussed in\cite[p.~343]{KS1}
and \cite[p.~58]{KS2}). Also, $X_t$ is a complex 
submanifold of $X$, so $\struct{X}|_{X_t}\,=\, \struct{X_t}$. We also have
the restriction map $r_t\,:\,\END[\struct{X}]{E}\,\longrightarrow\, \END[\struct{X_t}]{E_t}$.

Similarly, if $P\,:\, E\,\longrightarrow\, F$ is a first order $S$-differential operator,
where $F$ is a holomorphic vector bundle over $X$, then
the restriction map $r_t\,:\,E_t \,\longrightarrow\, F_t$ gives rise to a first order 
differential operator $P_t\,:\, E_t\,\longrightarrow\, F_t$ for every $t\, \in\, S$. Thus,
we have the restriction map $r_t\,:\,\DIFF[1]{S}{E}{F}\,\longrightarrow\, \DIFF[1]{\C}{E_t}{F_t}$.
In particular, for $E\,=\, F$, we have the restriction map 
$r_t\,:\,\DIFF[1]{S}{E}{E} \,\longrightarrow\, \DIFF[1]{\C}{E_t}{E_t}$ for every $t\,\in \,S$.
Since, the restriction of the relative tangent bundle $T(X/S)$ to each
fiber $X_t$ of $\pi$ is the tangent bundle $T(X_t)$ of the fiber
$X_t$, we have the restriction map $r_t\,:\,\cat{T}_{X/S}\,\longrightarrow\, \cat{T}_{X_t}$. 
 
Now, for every $t\,\in\, S$, the restriction maps gives a commutative diagram
\begin{equation}
\label{eq:cd2}
\xymatrix{
0 \ar[r] & \END[\struct{X}]{E} \ar[d]^{r_t} \ar[r] & \cat{A}t_S(E) 
\ar[d]^{r_t} \ar[r]^{\sigma_1} & \cat{T}_{X/S} \ar[d]^{r_t} \ar[r] & 0 \\
0 \ar[r] & \END[\struct{X_t}]{E_t} \ar[r] & \cat{A}t(E_t) \ar[r]^{\sigma_{1t}} & \cat{T}_{X_t} \ar[r] & 0 
}
\end{equation}
where the bottom sequence is the Atiyah sequence of the holomorphic vector
bundle $E_t$ over $X_t$ (see \eqref{eq:82}) and $\sigma_{1t}$ is the restriction
of the symbol map $\sigma_1$.

Suppose that $E$ admits a holomorphic $S$-connection, which is 
equivalent to saying that the relative Atiyah sequence in \eqref{eq:82} splits 
holomorphically. If $\nabla\,:\,\cat{T}_{X/S}\,\longrightarrow\,\cat{A}t_S(E)$ is 
a holomorphic splitting of the relative Atiyah sequence in \eqref{eq:82},
then for every $t \,\in\, S$, the restriction of
$\nabla$ to $\cat{T}_{X_t}$ gives an $\struct{X_t}$-module homomorphism
$\nabla_t\,:\, \cat{T}_{X_t} \,\longrightarrow\, \cat{A}t(E_t)$. Now, $\nabla_t$ is a 
holomorphic splitting of the Atiyah sequence of the holomorphic vector 
bundle $E_t$, which follows from the fact that the restriction maps $r_t$
defined above are surjective. Thus, we have the following:
 
\begin{proposition}\label{pro:3J}
Let $\pi\,:\, X \,\longrightarrow\, S$ be a surjective holomorphic proper
submersion with connected fibers and $\varpi\,:\, E \,\longrightarrow\, X$ a 
holomorphic vector bundle. Let $D$ be a holomorphic $S$-connection 
on $E$. Then for every $t \,\in\, S$, we have a holomorphic connection $D_t$
on the holomorphic vector bundle $E_t\,\longrightarrow\, X_t$. In other words, we have a
family $\{D_t\,\mid\, t \in S\}$ of holomorphic connections on the holomorphic 
family of vector bundles $\{E_t\,\longrightarrow\, X_t\,\mid\, t\in S\}$.
\end{proposition} 

\subsection{Smooth relative tangent bundle and smooth relative r-forms}\label{Smooth-rel-form}

As before, $\pi\,:\, X\,\longrightarrow\, S$ is a complex analytic family of
complex manifolds. Consider $\pi$ as a $C^\infty$ map between real manifolds.
We denote the smooth relative tangent bundle by $T^{\R}(X/S)$, while its
sheaf of smooth sections is denoted by $\cat{T}^{\R}_{X/S}$.
Similarly, there is a smooth relative cotangent bundle denoted by $A^1_{\R}(X/S)$ and
its sheaf of smooth sections, which is denoted by $\cat{A}^1_{\R}(X/S)$. Define
$$ T^{*}(X/S)_\C \,= \, A^1_{\R}(X/S) \otimes_\R \C \,=\, A^1_{\R}(X/S)_\C\, , $$
which is nothing but the complexification of the smooth relative 
cotangent bundle $A^1_{\R}(X/S)$.

A smooth section of $A^1_{\R}(X/S)_\C$ is called a \emph{complex valued smooth $1$-form on $X$ 
relative to $S$}, or a \emph{complex valued smooth relative $1$-form on $X$ over $S$}. We denote 
the sheaf of smooth sections of $A^1_{\R}(X/S)_\C$ by $\cat{A}^1_{X/S}$; also, denote the sheaf 
of complex valued smooth function on $X$ by $\cat{C}^\infty_{X}$. Then $\cat{A}^1_{X/S}$ is an 
$\cat{C}^\infty_{X}$-module, and there exists a unique $S$-derivation 
$d_{X/S}\,:\,\cat{C}^\infty_{X}\,\longrightarrow\, \cat{A}^1_{X/S}$. The kernel $\SKer{d_{X/S}}$ of $d_{X/S}$ is 
the sheaf of complex valued smooth functions on $X$, which are constant along the fibers $X_t$, 
for all $t \in S$, that is, $\SKer{d_{X/S}} \,=\, \pi^{-1}{\cat{C}^\infty_{S}}$, where 
${\cat{C}^\infty_{S}}$ is the sheaf of complex valued smooth functions on $S$.

Similarly, we can define the complex valued smooth relative $r$-forms on $X$ over $S$.
A smooth section of $\Lambda^rT^{*}(X/S)_\C$ is called a \emph{complex
valued smooth relative $r$-form} on $X$ over $S$. Denote the sheaf of
smooth sections of $\Lambda^rT^{*}(X/S)_\C$ by $\cat{A}^r_{X/S}$. The following
analog Theorem \ref{thm:1.5} is derived using Proposition \ref{pro:2} again.

\begin{theorem}\label{thm:2}
There exist canonical $S$-linear maps 
$\delta^r_{X/S}\,:\, \cat{A}^r_{X/S}\,\longrightarrow\, \cat{A}^{r+1}_{X/S}$ called relative exterior 
derivative satisfying the following:
\begin{enumerate}
\item \label{33} $\delta^0_{X/S} \,=\, d_{X/S}\,:\, \cat{C}^\infty_{X} \cat{A}^1_{X/S}$,

\item \label{34} $\delta^{r+1}_{X/S} \circ \delta^r_{X/S} \,=\, 0$, and

\item \label{35} $\delta_{X/S}(\alpha \wedge \beta)\,=\, \delta_{X/S} \alpha
\wedge \beta + (-1)^{r} \alpha \wedge \delta_{X/S}\beta$ for all local
sections $\alpha$ of $\cat{A}^r_{X/S}$ and $\beta$ of $\cat{A}^s_{X/S}$. 
\end{enumerate}
 \end{theorem}

\begin{proof}
First note that the sheaf 
$\DER[S]{\cat{C}^\infty_{X}}{\cat{C}^\infty_{X}}$ is
canonically isomorphic to the relative tangent sheaf
$\cat{T}^\R_{X/S}$ as an $\cat{C}^\infty_{X}$-module, and hence 
$\cat{A}lt^r_{\cat{C}^\infty_{X}}(\DER[S]{\cat{C}^\infty_{X}}{\cat{C
}^\infty_{X}},\,\cat{C}^\infty_{X})$ 
is canonically isomorphic to the sheaf $\cat{A}^r_{X/S}$ as an 
$\cat{C}^\infty_{X}$-module. Considering the canonical $S$-connection in 
$\cat{C}^\infty_{X}$, by Proposition \ref{pro:2}, there is a canonical
$S$-linear map that satisfies \eqref{33}, and by Proposition
\ref{pro:3}, it satisfies \eqref{35}. Finally, \eqref{34} follows using
By Corollary \ref{cor:2}\eqref{27}.
\end{proof}

Henceforth, we shall denote $\delta^r_{X/S}$ by $d_{X/S}$, for all $r \,\geq\, 0$.
By the relative Poincar\'e lemma, there is an exact sequence
$$
0\,\longrightarrow\,\pi^{-1}\cat{C}^\infty_{S}\,\longrightarrow\,\cat{C}^\infty_{X} \xrightarrow{d_{X/S}}
\cat{A}^1_{X/S} \xrightarrow{d_{X/S}} \cdots \xrightarrow{d_{X/S}}
\cat{A}^{2l}_{X/S} \,\longrightarrow\, 0
$$
of $\cat{C}^\infty_{X}$-module and $S$-linear maps. Thus we have a smooth relative de Rham complex 
$(\cat{A}^\bullet_{X/S},\, d_{X/S})$, which is a resolution of the sheaf $\pi^{-1}\cat{C}^\infty_{S}$.
 
For every integer $p \,\geq\, 0$ and for every open subset 
$V\,\subset\, S$, the assignment
$$V\,\longmapsto\, \hcoh{p}{\pi^{-1}(V)}{\cat{A}^\bullet_{X/S}|_{\pi^{-1}(V)}}$$ is a 
presheaf of $\pi_\ast \cat{C}^\infty_{X} (V) \,=\,\cat{C}^\infty_{X}(\pi^{-1}(V))$-module, where 
$\hcoh{q}{\pi^{-1}(V)}{\cat{A}^\bullet_{X/S}|_{\pi^{-1}(V)}}$ denotes the
hypercohomolgy group of $\pi^{-1}(V)\,\subset\, X$ with values in 
$\cat{A}^\bullet_{X/S}|_{\pi^{-1}(V)}$. The sheafification of this presheaf is a
$\cat{C}^\infty_{S}$-module, and it is denoted by $ \mathbb{R}^p\pi_\ast(\cat{A}^\bullet_{X/S})$.

The sheaf 
$\mathbb{R}^p\pi_\ast(\cat{A}^\bullet_{X/S})$ of $\cat{C}^\infty_{S}$-module is 
called the sheaf of relative de Rham cohomology, and it is denoted by 
$\drcoh{p}{dR}{X/S}$. Since, $\cat{A}^\bullet_{X/S}$ is an acyclic resolution of 
$\pi^{-1}\cat{C}^\infty_{S}$, we have the following:
 
\begin{proposition}
Let $\pi\,:\, X \,\longrightarrow\, S$ be a holomorphic proper submersion of complex manifolds
with connected fibers. Then $$\drcoh{p}{dR}{X/S} \,\cong\, R^p 
\pi_\ast(\pi^{-1}\cat{C}^\infty_{S})\, ,$$
where $R^p \pi_\ast(\pi^{-1}\cat{C}^\infty_{S})$ is the higher direct 
image sheaf of $\pi^{-1} \cat{C}^\infty_{S}$ on $S$.
\end{proposition}
 
\begin{proof}
Since, for each $p \,\geq\, 0$, the sheaf $\cat{A}^{p}_{X/S}$ is fine, from the definition of
hypercohomology, we have 
$$\hcoh{p}{\pi^{-1}(V)}{\cat{A}^\bullet_{X/S}|_{\pi^{-1}(V)}} \,\cong\, 
\gcoh{p}{\Gamma(\pi^{-1}(V),\cat{A}^\bullet_{X/S}|_{\pi^{-1}(V)})}$$
for every open subset $V \,\subset\, S$. 
Also, $\cat{A}^\bullet_{X/S}$ is an acyclic resolution of 
$\pi^{-1}\cat{C}^\infty_{S}$. Thus, we have
$$\coh{p}{\pi^{-1}(V)}{\pi^{-1}\cat{C}^\infty_{S}} \,\cong\, 
\gcoh{p}{\Gamma(\pi^{-1}(V),\cat{A}^\bullet_{X/S}|_{\pi^{-1}(V)} )}$$ for every
open subset $V \,\subset\, S$. 
Now, the proposition follows from the definition of 
higher direct image sheaves.
\end{proof}
 
Note that $\drcoh{p}{dR}{X/S}$ is locally free $\cat{C}^\infty_{S}$-module.
 
\begin{proposition}[{Pullback of smooth relative forms}]\label{prop:3.A.1} 
Suppose we have the following commutative diagram
 \begin{equation} 
\label{eq:3.1}
\xymatrix{Y \ar[d]^{\pi'} \ar[r]^f & X \ar[d]^\pi \\
T \ar[r]^g & S \\}
\end{equation} 
of complex manifolds and holomorphic maps, where $\pi$, and $\pi'$ are 
surjective holomorphic proper submersions. Then, for every open subset 
$U \,\subset\, X$, and every smooth relative differential form $\omega\,\in\,
\cat{A}^r_{X/S}(U)$, the pullback $\widetilde{f^*}(\omega)$ is an element 
of $ \cat{A}^r_{Y/T}(f^{-1}(U))$.
\end{proposition}

\begin{proof}
 Given the commutative diagram in \eqref{eq:3.1}, we have the following 
 commutative diagrams:
 \begin{equation}
 \label{eq:3.2}
 \xymatrix{ 0 \ar[r] & \pi^*\cat{A}_S^1 \ar[d]^{\pi*g*} \ar[r] & 
 \cat{A}^1_{X} \ar[d]^{f^*} \ar[r] & \cat{A}^1_{X/S} \ar[d]^{\widetilde{f^*}} 
 \ar[r] & 0 \\
 0 \ar[r] & \pi^*\cat{A}_T^1 \ar[r] & \cat{A}_Y^1 \ar[r] & \cat{A}_{Y/T}^1 
 \ar[r] & 0 \\}
 \end{equation}
 
 \begin{equation}
 \label{eq:3.3}
 \xymatrix{ 0 \ar[r] & \pi^*\cat{A}_S^1 \otimes_{\cat{C}^\infty_{X,\C}} 
 \cat{A}_X^{r-1} \ar[d]^{\pi*g*} \ar[r] & 
 \cat{A}^r_{X} \ar[d]^{f^*} \ar[r] & \cat{A}^r_{X/S} \ar[d]^{\widetilde{f^*}} 
 \ar[r] & 0 \\
 0 \ar[r] & \pi^*\cat{A}_T^1 \otimes_{\cat{C}^\infty_{Y,\C}} \cat{A}^{r-1}_Y 
 \ar[r] & \cat{A}_Y^r \ar[r] & \cat{A}_{Y/T}^r \ar[r] & 0 \\}
 \end{equation}
Thus, given any $\omega \,\in\, \cat{A}^1_{X/S}(U)$, from \eqref{eq:3.2},
we get that $\widetilde{f^*}(\omega)$ is an element of $\cat{A}^1_{Y/T}(f^{-1}(U))$,
and similarly, from \eqref{eq:3.3} it follows that for any smooth relative $r$-form
$\omega \,\in\, \cat{A}^r_{X/S}(U)$, we have $\widetilde{f^*}(\omega) 
\,\in\, \cat{A}^r_{Y/T}(f^{-1}(U))$. 
\end{proof}
 
\subsection{Smooth relative connection and relative Chern class}\label{rel-chern}

In this section, we define the relative Chern class of a differentiable
family of complex vector bundles $\varpi\,:\,E \,\longrightarrow\, X$ of rank $r$. For each
$t\, \in\, S$, the restriction of $E$ to $X_t\,=\, \pi^{-1}(t)$ will be denoted by $E_t$.

We follow Section \ref{cc-mt} and substitute $E$ in place of $\cat{F}$ there.
Let $D$ be a smooth $S$-connection on $E$. Let $(U_\alpha,\,h_\alpha)$ be a trivialization
of $E$ over $U_\alpha\,\subset\, X$.
Let $R$ be the $S$-curvature form for $D$, and let $\Omega_\alpha\,=\, (\Omega_{ij\alpha})$ 
be the curvature matrix of $D$ over $U_\alpha$, 
as defined in Section \ref{cc-mt}, so
$\Omega_{ij\alpha}\,\in\,\cat{A}^2_{X/S}(U_\alpha)$. We have 
$\Omega_\beta\,=\, g_{\alpha \beta}^{-1} \Omega_{\alpha} g_{\alpha \beta}$, 
where $g_{\alpha \beta}\,:\, U_{\alpha} \cap U_{\beta}\,\longrightarrow\, {\rm GL}_r(\C)$ is the 
change of frame matrix (transition function), which is a smooth map.

Consider the adjoint action of ${\rm GL}_r(\C)$ on it Lie algebra $\mathfrak{gl}_r(\C)\,=\,
{\rm M}_r(\C)$. Let $f$ be a ${\rm GL}_r(\C)$-invariant homogeneous 
polynomial on $\mathfrak{gl}_r(\C)$ of degree $p$. Then, we can associate a unique
$p$-multilinear symmetric map $\widetilde{f}$ on $\mathfrak{gl}_r(\C)$ 
such that $f(X)\,=\, \widetilde{f}(X,\cdots ,X)$, for all $X \in \mathfrak{gl}_r(\C)$.
Define $$\gamma_\alpha \,=\, \widetilde{f}(\Omega_\alpha,\,\cdots, \,\Omega_\alpha)
\,=\, f(\Omega_\alpha)\,\in\,
\cat{A}^{2p}_{X/S}(U_\alpha)\, .$$
Since $f$ is ${\rm GL}_r(\C)$-invariant, it follows that $\gamma_\alpha$ is independent of
the choice of frame, and hence it
defines a global smooth relative differential form of degree $2p$, which we denote by the symbol
$\gamma \in \cat{A}^{2p}_{X/S}(X)$.

\begin{theorem}\label{thm:5}
Let $\pi\,:\, X\,\longrightarrow\,S$ be a surjective holomorphic proper 
submersion of complex manifolds with connected fibers and $\varpi\,:\, E \,\longrightarrow\, X$ a differentiable
family of complex vector bundle. Let $D$ be a smooth $S$-connection on $E$.
Suppose that $f$ is a ${\rm GL}_r(\C)$-invariant 
polynomial function on $\mathfrak{gl}_r(\C)$ of degree $p$. Then the following hold:
\begin{enumerate}
\item $\gamma\,=\, f(\Omega_\alpha)$ is $d_{X/S}$-closed, that is,
$d_{X/S}(\gamma) \,=\, 0$.

\item The image $[\gamma]$ of $\gamma$ in the relative de Rham cohomology group
$$\gcoh{2p}{\Gamma(X,\,\cat{A}^\bullet_{X/S} )} \,=\, 
\coh{2p}{X}{\pi^{-1}\cat{C}^\infty_{S}}$$
is independent of the smooth $S$-connection $D$ on $E$.
\end{enumerate}
\end{theorem}

\begin{proof}
This is proved in \cite{K} (Chapter II, Section 2, p.~36).
\end{proof}

Define homogeneous polynomials $f_p$ on $\mathfrak{gl}_r(\C)$, of degree 
$p\,=\, 1,\,2,\,\cdots ,\,r$, to be the coefficient of $\lambda^p$ in the following 
expression:
$$
\mathrm{det}(\lambda \mathrm{I} + \frac{\sqrt{-1}}{2 \pi} A) \,=\, \Sigma_{j=0}^{r}
\lambda^{r-j} f_j(\frac{\sqrt{-1}}{2 \pi} A)\, ,
$$
where $f_0(\frac{\sqrt{-1}}{2 \pi} A) \,=\, 1$ while $f_r(\frac{\sqrt{-1}}{2 \pi} A)$ is the 
coefficient of $\lambda^0$. These polynomials $f_1,\,\cdots,\, f_r$ are $GL_r(\C)$-invariant, and
they generate the algebra of ${\rm GL}_r(\C)$-invariant polynomials on 
$\mathfrak{gl}_r(\C)$. We now define the \emph{p-th cohomology class} as follows:
$$
c^{S}_p(E)\,=\, [f_p(\frac{\sqrt{-1}}{2 \pi} \Omega)]\,\in\, \gcoh{2p}{\Gamma(X,\,
\cat{A}^\bullet_{X/S})}
$$
for $p \,=\,0,\,1,\,\cdots ,\,r$. 
 
The relative de Rham cohomology sheaf $\drcoh{p}{dR}{X/S} \,\cong\, 
R^p \pi_\ast(\pi^{-1}\cat{C}^\infty_{S})$ on $S$ is by 
definition the sheafification of the presheaf $V\,\longmapsto\, 
\coh{p}{\pi^{-1}(V)}{\pi^{-1}\cat{C}^\infty_{S}|_{\pi^{-1}(V)}}$,
for open subset $V \,\subset\, S$. Therefore, we have a natural homomorphism
$$
\rho\,:\, \coh{2p}{X}{\pi^{-1}\cat{C}^\infty_{S}} \,\longrightarrow\,
\drcoh{2p}{dR}{X/S}(S)
$$
which maps $c^{S}_p(E)$ to $\rho( c^S_p(E))\,\in\, \drcoh{2p}{dR}{X/S}(S)$.

Define $C^S_p(E) \,=\, \rho(c^S_p(E))$. We call $C^S_p(E)$ the \emph{p-th 
relative Chern class of E over S}. 
Let $$C^S(E)\,=\, \sum_{p \geq 0} C^S_p(E)\,\in\, \drcoh{*}{dR}{X/S}(S)\,=\, 
\oplus_{k \geq 0} \drcoh{k}{dR}{X/S}(S)$$ be the \emph{total relative 
Chern class} of $E$.

\begin{proposition}
\label{lem:3K1} Let $E \,\xrightarrow{\varpi}\, X \,\xrightarrow{\pi}\, S$ be as in
Theorem \ref{thm:5}. Let $\pi'\,:\,Y \,\longrightarrow\, T$ be a surjective holomorphic proper 
submersion, such that the following diagram
\begin{equation}
\label{eq:5.91}
\xymatrix{Y \ar[d]^{\pi'} \ar[r]^f & X \ar[d]^\pi \\
T \ar[r]^g & S \\}
\end{equation} 
is commutative, where $f\,:\, Y \,\longrightarrow\, X$ and $g\,:\, T
\,\longrightarrow\, S$ are holomorphic maps. Then 
$$
f^*(C^S(E))\,= \, C^T(f^*E)\, ,
$$
where $C^S(E)$ is the total relative Chern class of $E$ over $S$, and
$C^T(f^*E)$ is the total relative Chern class of $f^*E$ over $T$. 
\end{proposition}

\begin{proof}
Let $D$ be smooth $S$-connection in $E$. It is enough to define a smooth 
$T$-connection $D^*$ in $f^*E$, such that $f^*\Omega \,=\, \Omega^*$, where $
\Omega^*$ is the curvature matrix of $D^*$. Let $e \,=\, (e_1,\,\cdots,\,e_r)$ be a
frame of $E$ over an open subset $U$ of $X$. Then, we have
$e^* \,=\, (e^*_1,\,\cdots,\, e^*_r)$, where $e^*_i \,=\, f \circ e_i^*\,:\, f^{-1}(U)
\,\longrightarrow\, E$ is a frame of $f^*E$ over $f^{-1}(U)$. If $a\,:\, U
\,\longrightarrow\, {\rm GL}_r(\C)$ is a 
change of frame over $U$, then $$f^*a \,=\, a^* \,=\, a \circ f\,:\, f^{-1}(U)\,\longrightarrow\,
{\rm GL}_r(\C)$$ is a change of frame in $f^*E$ over $f^{-1}(U)$. Now, we define 
$S$-connection matrix 
$$
\omega^* \,=\, f^*\omega \,=\, [\widetilde{f^*} \omega_{ij}]\, ,
$$
where, $\omega_{ij}\,\in\, \cat{A}_{X/S}^1(U)$, and $\widetilde{f^*}\,: \,\cat{A}_{X/S}^1
\,\longrightarrow\,\cat{A}_{Y/T}^1$ is the pullback map of the relative forms as in
Proposition \ref{prop:3.A.1}. Moreover, if $\omega'$ is the connection matrix with respect
to the frame $e' \,=\, e.a$ and $\omega'^*$ is the pullback of $\omega'$ under
$f^*$, then 
$$
\omega'^* \,=\, a^{*{-1}} \omega^* a^* + a^{*{-1}} da^*\, .
$$
Thus, if we consider $D^* \,=\, d_{Y/T} + \omega^*$, then from above 
compatibility condition, this defines a smooth $T$-connection in $f^*E$.
Let $\Omega^*$ be the curvature form of $D^*$. Then
$$
\Omega^*\,=\, d_{Y/T} \omega^* + \omega^* \wedge \omega^* 
\, =\, f^* \Omega\, .
$$
Now, consider the homogeneous polynomial $f_p$ of degree $p$ as defined 
above. The $p$-th cohomology class is
$$
c_p^T(f^*E)\,=\, [f_p(\frac{\sqrt{-1}}{2 \pi} \Omega^*)] 
\,=\, f^*[c_p^S(E)] 
$$
for all $p \,\geq \, 0$, which is the pullback of the cohomology class 
$[c_p^S(E)]$, where 
$$f^*\,:\,\gcoh{2p}{\Gamma(X,\cat{A}^\bullet_{X/S} )} \,\longrightarrow\, 
\gcoh{2p}{\Gamma(Y,\cat{A}^\bullet_{Y/T})}$$ is the morphism of $\C$-vector 
spaces induced by the commutative diagram \eqref{eq:5.91}.
Further, we have the following commutative diagram
\begin{equation}
\label{eq:5.94}
\xymatrix{ \coh{2p}{X}{\pi^{-1}\cat{C}^\infty_{S}} \ar[d]^{f^*} 
\ar[r]^\rho & \drcoh{2p}{dR}{X/S}(S) \ar[d]^{f^*} \\
\coh{2p}{Y}{\pi^{-1}\cat{C}^\infty_{T}} \ar[r]^\rho
& \drcoh{2p}{dR}{Y/T}(T)}
\end{equation} 
which implies that $f^*(C_p^S(E))\,=\,C_p^T(f^*E)$. This completes the proof.
\end{proof}

In particular, taking $T\,= \,\{t\}\,\subset \,S$, $g$ to be the inclusion map
$t \,\hookrightarrow\, S$, $Y \,=\, X_t$, $\pi' \,=\, \pi|_{X_t}\,:\, X_t
\,\longrightarrow\, T$ and $f$ to be the inclusion map $j\,:\,X_t\,\hookrightarrow\, X$,
by Proposition \ref{lem:3K1}, we have the following:

\begin{corollary}\label{cor:5.1}
For every $t \,\in\, S$, there is a natural map $$j^* \,:\,\drcoh{2p}{dR}{X/S}(S)
\,\longrightarrow\,\coh{2p}{X_t}{\C}$$ which maps the $p$-th relative Chern class of $E$ to
the $p$-th Chern class of the smooth vector bundle $E_t \,\longrightarrow\, X_t$, that is,
$j^*(C_p^S(E)) \,=\, c_p(E_t)$.
\end{corollary}

The following topological proper base change theorem is given in \cite[p.~202, Remark 4.17.1]{G}
and \cite[p.~19, Corollary 2.25]{D}:

\begin{theorem}[{Topological proper base change}]\label{thm:3.13}
Let $f\,:\, X \, \longrightarrow\, S$ be a proper continuous map of Hausdorff topological spaces. 
Suppose that $S$ is locally compact, and $\cat{F}$ is a sheaf of abelian 
groups on $X$. Then for all $t \,\in \,S$, we have a canonical isomorphism
$$
(R^pf_*\cat{F})_t \,\simeq\, \coh{p}{f^{-1}(t)}{\cat{F}|_{f^{-1}(t)}}
$$
of abelian groups.
\end{theorem}

Note that $\drcoh{p}{dR}{X/S}$ is a locally free 
$\cat{C}^\infty_{S}$-module, and from Theorem \ref{thm:3.13}
we have a $\C$-vector space isomorphism
\begin{equation}
\label{eq:3.I.1}
\eta\,:\, \drcoh{p}{dR}{X/S}_t \otimes_{\cat{C}^\infty_{{S},t}} k(t)
\, \longrightarrow\, \coh{p}{X_t}{\C}
\end{equation}
for every $t \,\in\, S$.

\begin{theorem}\label{thm:3.18}
Let $\pi\,:\, X \, \longrightarrow\, S$ be a surjective holomorphic proper
submersion with connected fibers, such that $\pi^{-1}(t) \,=\, X_t$ is compact
K\"ahler manifold for every $t\, \in\, S$.
Let $\varpi\, :\, E \, \longrightarrow\, X$ be a holomorphic vector 
bundle. Suppose that $E$ admits a holomorphic $S$-connection.
Then all the relative Chern classes $C_p^S(E) \,\in\, \drcoh{2p}{dR}{X/S}(S)$ 
of $E$ over $S$ are zero.
\end{theorem}
 
\begin{proof}
Let $D$ be a holomorphic $S$-connection on $E$. From Proposition
\ref{pro:3J} it follows that for every $t \,\in\, S$, there is a holomorphic connection $D_t$ 
in $E_t$. Since $X_t$ is a compact complex manifold of
K\"ahler type, from Theorem \textbf{4} in \cite[p.~192]{A} it follows that
all the Chern classes $c_p(E_t)$ of $E_t$ are zero. From Corollary \ref{cor:5.1} and
the isomorphism in \eqref{eq:3.I.1} we have the following commutative diagram;
$$
\xymatrix{\drcoh{2p}{dR}{X/S}(S) \ar[rd]_{j^*} \ar[r] & 
\drcoh{p}{dR}{X/S}_t \otimes_{\cat{C}^\infty_{{S},t}} k(t) 
\ar[d]^{\eta} \\
& \coh{2p}{X_t}{\C} \\ }
$$
Now,
$$
\eta(C_p^S(E)_t \otimes 1)\,=\, j^*(C_p^S(E)) \,=\, c_p(E_t) \,=\, 0\, ,
$$
which implies that $C_p^S(E)_t \otimes 1 \,=\, 0$, for every $t \,\in\, S$,
because $\eta$ is an isomorphism. Thus, we have $C_p^S(E)\, =\, 0$. This completes the proof.
\end{proof}

\subsection{A sufficient condition for holomorphic connection}\label{suff}

Given a surjective holomorphic proper submersion $\pi \,: X \,\longrightarrow\,S$
with connected fibers and 
a holomorphic vector bundle $\varpi\,:\, E \,\longrightarrow\, X$, Proposition 
\ref{pro:3J} gives a necessary condition for the existence of a holomorphic $S$-connection
on $E$, namely the vector bundle $E_t\,=\, E\vert_{X_t} \,\longrightarrow\, X_t$ should
admit a holomorphic connection for every $t\, \in\, S$.

If for every each $t\, \in\, S$ the vector bundle $E_t$ admits a holomorphic connection, it
is natural to ask whether $E$ admits a holomorphic $S$-connection.
We will present a sufficient condition for the existence of holomorphic $S$-connection on $E$.

\begin{theorem}\label{thm:3l}
Let $E \,\stackrel{\varpi}{\longrightarrow}\, X$
be a holomorphic vector bundle. Suppose that for every
$t \in S$, there is a holomorphic connection on the holomorphic vector bundle
$\varpi|_{E_t}\, :\, E_t \,\longrightarrow\, X_t$, and
$$\coh{1}{S}{\pi_*(\Omega^1_{X/S}(\END[\struct{X}]{E}))} \,=\, 0\, .$$
Then, $E$ admits a holomorphic $S$-connection.
\end{theorem}

\begin{proof}
Consider the relative Atiyah exact sequence in \eqref{eq:82}. Tensoring it by
$\Omega^1_{X/S}$ produces the exact sequence
\begin{equation}\label{q}
0\,\longrightarrow\, \Omega^1_{X/S}(\END[\struct{X}]{E})\,\longrightarrow\, \Omega^1_{X/S}(\cat{A}t_S(E))
\,\stackrel{q}{\longrightarrow}\,\Omega^1_{X/S}\otimes \cat{T}_{X/S}
\,\longrightarrow\, 0\, .
\end{equation}
Note that $\struct{X}\cdot \text{Id} \, \subset\, \text{End}(\cat{T}_{X/S})\,=\,
\Omega^1_{X/S}\otimes \cat{T}_{X/S}$. Define
$$
\Omega^1_{X/S}(\cat{A}t'_S(E))\, :=\, q^{-1}(\struct{X}\cdot \text{Id})\, \subset\,
\Omega^1_{X/S}(\cat{A}t_S(E))\, ,
$$
where $q$ is the projection in \eqref{q}. So we have the short exact sequence of sheaves
\begin{equation}\label{q2}
0\,\longrightarrow\, \Omega^1_{X/S}(\END[\struct{X}]{E})\,\longrightarrow\, \Omega^1_{X/S}(\cat{A}t'_S(E))
\,\stackrel{q}{\longrightarrow}\,\struct{X}\,\longrightarrow\, 0
\end{equation}
on $X$, where $\Omega^1_{X/S}(\cat{A}t'_S(E))$ is constructed above. Let
\begin{equation}\label{q3}
\Phi\, :\, {\mathbb C}\, =\, {\rm H}^0(X,\, \struct{X}\cdot \text{Id})\, \longrightarrow\,
{\rm H}^1(X,\, \Omega^1_{X/S}(\END[\struct{X}]{E}))
\end{equation}
be the homomorphism in the long exact sequence of cohomologies associated to the exact sequence
in \eqref{q2}. The relative Atiyah class $\at[S]{E}$ (see Corollary \ref{cor:4}) coincides
with $\Phi(1)\, \in\, {\rm H}^1(X,\, \Omega^1_{X/S}(\END[\struct{X}]{E}))$. Therefore, from
Corollary \ref{cor:4} it follows that 
$E$ admits a holomorphic $S$-connection if and only if
\begin{equation}\label{q4}
\Phi(1)\, =\,0\, .
\end{equation}

To prove the vanishing statement in \eqref{q4}, first note that
${\rm H}^1(X,\, \Omega^1_{X/S}(\END[\struct{X}]{E}))$ fits in the exact sequence
\begin{equation}\label{q5}
{\rm H}^1(S,\, \pi_*(\Omega^1_{X/S}(\END[\struct{X}]{E})))\, \stackrel{\beta_1}{\longrightarrow}\,
{\rm H}^1(X,\, \Omega^1_{X/S}(\END[\struct{X}]{E}))\, \stackrel{q_1}{\longrightarrow}\,
{\rm H}^0(S,\, R^1\pi_*(\Omega^1_{X/S}(\END[\struct{X}]{E})))\, ,
\end{equation}
where $\pi$ is the projection of $X$ to $S$.

The given condition that for every
$t \in S$, there is a holomorphic connection on the holomorphic vector bundle
$\varpi|_{E_t}\, :\, E_t \,\longrightarrow\, X_t$, implies that
$$
q_1(\Phi(1))\,=\, 0\, ,
$$
where $q_1$ is the homomorphism in \eqref{q5}. Therefore, from the exact sequence in
\eqref{q5} we conclude that
$$
\Phi(1)\, \in\, \beta_1({\rm H}^1(S,\, \pi_*(\Omega^1_{X/S}(\END[\struct{X}]{E}))))\, .
$$
Finally, the given condition that ${\rm H}^1(S,\, \pi_*(\Omega^1_{X/S}(\END[\struct{X}]{E})))\,=\, 0$.
implies that $\Phi(1)\, =\, 0$. Since \eqref{q4} holds, the vector bundle
$E$ admits a holomorphic $S$-connection.
\end{proof}

Take $\pi\,:\, X \,\longrightarrow\, S$ to be a surjective holomorphic proper 
submersion of relative dimension one, so $\pi^{-1}(t)$ is a compact connected Riemann surface, for 
every $t \in S$. Then, by the \emph{Atiyah-Weil} criterion given in 
\cite{A}, \cite{W} and \cite{BR} (Theorem \textbf{6.12}), we have the following:

\begin{corollary}\label{cor:3.21}
Let $\pi\,:\, X \,\longrightarrow\, S$ be a surjective holomorphic proper 
submersion such that $\pi^{-1}(t) \,=\, X_t$ is a compact connected Riemann surface for
every $t \,\in \,S$. Let $\varpi\,:\, E \,\longrightarrow\, X$ be a holomorphic vector bundle.
Suppose that for every $t \,\in\, S$, the degree of the indecomposable 
components of $E_t$ are zero and
$$\coh{1}{S}{\pi_*(\Omega^1_{X/S}(\END[\struct{X}]{E}))} \,=\, 0\, .$$
Then, $E$ admits a holomorphic $S$-connection.
\end{corollary} 

\section*{Acknowledgement}

We thank the referee for helpful advice. The first 
author is supported by a J. C. Bose Fellowship. 
The second author is very grateful to his Ph.D. advisor 
Prof. N. Raghavendra for introducing this topic  and 
useful discussions. The 
second author would like to acknowledge the Department 
of Atomic Energy, Government of India for 
providing the research grant.

\end{document}